\documentclass[reqno]{amsart}
\usepackage{amsmath}
\usepackage{amssymb}
\usepackage{amsthm}
  \usepackage{paralist}
  \usepackage{graphics} %% add this and next lines if pictures should be in esp format
  \usepackage{epsfig} %For pictures: screened artwork should be set up with an 85 or 100 line screen
\usepackage{graphicx}  \usepackage{epstopdf} %This is to transfer .eps figure to .pdf figure; please compile your paper using PDFLaTex or PDFTeXify.
 \usepackage[colorlinks=true]{hyperref}
\hypersetup{urlcolor=blue, citecolor=red}
\usepackage{comment}

\newtheorem{theorem}{Theorem}[section]

\newtheorem{lemma}[theorem]{Lemma}

\theoremstyle{definition}
\newtheorem{definition}[theorem]{Definition}

\title[A free boundary problem with migration into rubber] % Running head is the full title or shortened version of the full title. This will appear at the top of odd pages. Please make sure it fits within the width limit.
      {Local existence of a solution to a free boundary problem describing migration into rubber with a breaking effect} % Only the first word and proper nouns should be capitalized

\author[Kota Kumazaki, Toyohiko Aiki and Adrian Muntean]{}


 \email{k.kumazaki@nagasaki-u.ac.jp}
 \email{aikit@fc.jwu.ac.jp}
 \email{adrian.muntean@kau.se}

\begin{document}
\maketitle

% Enter the first author's name and address:
\centerline{\scshape Kota Kumazaki$^*$}
\medskip
{\footnotesize
% Enter the address of the first author
 \centerline{Nagasaki University}
   \centerline{Faculity of Education}
   \centerline{1-14 Bunkyo-cho, Nagasaki-city, Nagasaki, 851-8521, Japan}
} % Do not forget to end {\footnotesize with the sign }

\medskip

\centerline{\scshape Toyohiko Aiki}
\medskip
{\footnotesize
 % Enter the address of the second and third authors
 \centerline{Japan Women's University}
   \centerline{Department of Mathematics Faculty of Scicence}
   \centerline{2-8-1, Mejirodai, Bunkyo-ku, Tokyo, 112-8681, Japan }
}

\medskip

\centerline{\scshape Adrian Muntean}
\medskip
{\footnotesize
 % Enter the address of the second and third authors
 \centerline{Karlstads University}
   \centerline{Department of Mathematics and Computer Science}
   \centerline{Universitetsgatan 2, 651 88 Karlstad, Sweden}
}

\bigskip

% The name of the associate editor will be entered by AIMS editorial staff.
% "Communicated by the associate editor name" is not needed for special issue.
% \centerline{(Communicated by the associate editor name)}

%\title{Local existence of a solution to a free boundary problem describing migration into rubber with a breaking effect}

%\author{Kota Kumazaki\\
%Faculty of Education, Nagasaki University,\\
%1-14 Bunkyo-cho, Nagasaki-city, Nagasaki, 851-8521, Japan\\
%k.kumazaki@nagasaki-u.ac.jp\\
%Toyohiko Aiki\\
%Department of Mathematics Faculty of Sciences, Japan Women's University,\\
%2-8-1, Mejirodai, Bunkyo-ku, Tokyo, 112-8681, Japan \\
%aikit@fc.jwc.ac.jp\\
%Adrian Muntean\\
%Department of Mathematics and Computer Science, Karlstads University,\\
%Universitetsgatan 2, 651 88 Karlstad, Sweden\\
%adrian.muntean@kau.se
%}

%The abstract of your paper
\begin{abstract}
We consider a one-dimensional free boundary problem describing the migration of diffusants into rubber. 
In our setting, the free boundary represents the position of the front delimitating the diffusant region. The growth rate of this region is described by an ordinary differential equation that includes the effect of breaking the growth of the diffusant region. In this specific context, the breaking mechanism is assumed to be the hyperelastic response to a too fast diffusion penetration.

In recent works, we considered a similar class of free boundary problems modeling diffusants penetration in rubbers, but without attempting to deal with the possibility of breaking or accelerating the occurring free boundaries.  For simplified settings, we were able to show the global existence and uniqueness as well as the large time behavior of the corresponding solutions to our formulations.  Since here the breaking effect is contained in the free boundary condition, our previous results are not anymore applicable. The main mathematical obstacle in ensuring the existence of a solution is the non-monotonic structure of the free boundary. 

In this paper, we establish the existence and uniqueness of a weak solution to the free boundary problem with breaking effect and give explicitly the maximum value that the free boundary can reach.
\end{abstract}

%The title of your section 1
\section{Introduction}
\label{intro}
Both natural rubber and synthetic elastomers are widely used in engineering applications, such as connecting components for offshore wind farms that are nowadays intensively used to harvesting energy. Due to the growing importance of good theoretical estimations of durability properties of elastomeric parts, a well-trusted modelling of the material behaviour is an essential prerequisite so that numerical prediction are available for large times.
Rubber-like materials exhibit a strongly non-linear behaviour characterized by a joint large strain and a non-linear stress-strain response \cite{Chagnon}. In the presence of aggressive environments  (like those where large bodies of water are present), the nonlinear behavior of elastomers, which can be in fact studied in controlled laboratory conditions,  alters non-intuitively \cite{Brunier}. The main reason for this situation is that the porous structure of rubber-like materials allows small particles to diffuse inside. Ingressing particles, often ionically charged,  accumulate and interact with the internal solid fabric. Such interactions lead to unwanted alterations of the originally designed mechanical behavior; see e.g. \cite{Ren} for more context on this physical problem. From the modeling and mathematical upscaling point of view, the derivation of the correct model equations for such problem setting is open. 

In our recent works \cite{AKM, KA}, and \cite{KA2}, we considered a class of one-dimensional free boundary problems to model the diffusants penetration in dense and foam rubbers. The use of kinetic-type free boundary conditions is meant to avoid the explicit description of the mechanics of the involved materials and how the constitutive laws (in particular, the stress tensor) are affected by the presence of diffusants. One can use such free boundary formulations to derive theoretically-grounded practical estimations of the position of the diffusants penetration front and compare them with laboratory experiments; see e.g. \cite{NMKAMWG} where a parameter identification exercise has been performed in this context. 

The main novelty introduced in the present paper is that the growth rate of the region filled with diffusants is described by an ordinary differential equation that includes the effect of breaking its growth. The breaking mechanism should be seen here as the hyperelastic response to a too fast diffusion penetration. Such mechanism introduces non-monotonicity in our problem formulation. This makes us wonder whether this type of problems has a chance to be well-posed and, if true,  in which sense?  

To address the question, we consider the following one-dimensional free boundary problem (\ref{1-1})-(\ref{1-5}). Let $t\in [0, T]$ be the time variable for $T>0$. The problem is to find a curve $z=s(t)$ on $[0, T]$ and a function $u$ defined on the set $Q_s(T):=\{(t, z) | 0<t<T, \ 0<z<s(t) \}$ such that the following system of equations is satisfied: 
\begin{align}
& u_t-u_{zz}=0 \mbox{ for }(t, z)\in Q_s(T), \label{1-1} \tag{1.1} \\
& -u_z(t, 0)=\beta(b(t)-\gamma u(t, 0)) \mbox{ for }t\in(0, T), \label{1-2} \tag{1.2}\\
& -u_z(t, s(t))=\sigma(u(t, s(t)))s_t(t) \mbox{ for }t\in (0, T), \label{1-3} \tag{1.3} \\
& s_t(t)=a_0 (\sigma(u(t, s(t)))-\alpha s(t)) \mbox{ for }t\in (0, T), \label{1-4} \tag{1.4} \\
& s(0)=s_0, u(0, z)=u_0(z) \mbox{ for }z \in [0, s_0], \label{1-5} \tag{1.5}
\end{align}
where $\beta$, $\gamma$, $a_0$ and $\alpha$ are given positive constants, $b$ is a given function on $[0, T]$ and $s_0$ and $u_0$ are the initial data. In (\ref{1-3}) and (\ref{1-4}), $\sigma$ is the function on $\mathbb{R}$ called the positive part, i.e. it is given by 
\begin{equation}
\sigma(r)=
\begin{cases}
r \quad \mbox{ if }r\geq 0,\\
0 \quad \mbox{ if }r<0.
\end{cases} \tag{1.6}
\end{equation}

%Denote the problem \{(\ref{1-1})-(\ref{1-5})\} by $(\mbox{P})(u_0, s_0, b)$. 
The problem (\ref{1-1})-(\ref{1-5}), which is denoted here by $(\mbox{P})(u_0, s_0, b)$, was proposed by S. Nepal et al. in \cite{NMKAMWG} as a mathematical model describing the migration of diffusants into rubber. The set $[0,s(t)]$ represents the region occupied by a solvent occupying the one-dimensional pore $[0,\infty)$, where $s = s(t)$ is the position of moving interface of the region and $u = u(t, z)$ represents the content of the diffusant at the place $z \in [0,s(t)]$ at time $t>0$. %In \cite{NMKAMWG} the behavior of the free boundary $s$ is discussed numerically.

%Denote the problem \{(\ref{1-1})-(\ref{1-5})\} by $(\mbox{P})(u_0, s_0, b)$. 
From the mathematical point of view, $(\mbox{P})(u_0, s_0, b)$ is a one-phase free boundary problem with Robin-type boundary conditions at both boundaries. Such kind of problem structure has been considered in \cite{KA, KA2}, and \cite{AKM}. In our recent work \cite{AKM}, as a simplified setting for $(\mbox{P})(u_0, s_0, b)$, we consider the case $\alpha=0$ and prove the existence and uniqueness of a globally-in-time solution. Furthermore, what concerns the large time behavior of a solution, we show that the free boundary $s$ goes to infinity as time elapses.  
%For the case $\alpha>0$,  
%the following one dimensional free boundary problem which is a mathematical model describing water swelling in porous materials is considered in \cite{KA, KA2}: 

In this paper, we consider the case $\alpha >0$. As already anticipated,  the difficulty in this case is the lack of the monotonicity of the boundary condition imposed on the moving boundary.
%from (\ref{1-4}) it is not clear that the growth rate $s_t$ of the free boundary $s$ is non-negative. 
%Due to this, 
%it is difficult to prove the existence of a solution of $(\mbox{P})(u_0, s_0, b)$. 
%since the stopping effect $-\alpha s(t)$ in (\ref{1-4}) works with the glowth of the free boundary $s$, we infer that the free boundary $s$ keeps a finite length. However, 
Indeed, similarly to the proof of the existence result in \cite{AKM}, %in arguing the existence of solutions, %we fix the moving domain. 
%For $(\mbox{P})(u_0, s_0, b)$ 
by introducing a variable $\tilde{u}(t, y)=u(t, ys(t))$ for $(t, y)\in Q(T):=(0, T)\times (0, 1)$ we transform $(\mbox{P})(u_0, s_0, b)$ into a problem on the fixed domain $Q(T)$. Let denote this problem by $(\mbox{PC})(\tilde{u}_0, s_0, b)$, where $\tilde{u}_0(y)=u_0(s_0y)$ for $y\in [0, 1]$. 
%\begin{align*}
%& \tilde{u}_t(t, y)-\frac{1}{s^2(t)}\tilde{u}_{yy}(t, y)=\frac{ys_t(t)}{s(t)}\tilde{u}_y(t, y) \mbox{ for }(t, y)\in Q(T), \\ %\label{1-6} \\
%& -\frac{1}{s(t)}\tilde{u}_y(t, 0)=\beta(b(t)-\gamma \tilde{u}(t, 0)) \mbox{ for }t\in(0, T), \\ %\label{1-7} \\
%& -\frac{1}{s(t)}\tilde{u}_y(t, 1)=a_0 \sigma(\tilde{u}(t, 1))s_t(t)\mbox{ for }t\in (0, T),  \\ %\label{1-8} \\
%& s_t(t)=a_0\sigma(\tilde{u}(t, 1))(\sigma(\tilde{u}(t, 1))-\alpha s(t))  \mbox{ for }t\in (0, T), \\ %\label{1-9}\\
%& \tilde{u}(0, y)=u_0(y s(0))(:=\tilde{u}_0(y)) \mbox{ for }y \in [0, 1]. %\label{1-10}
%\end{align*}
%For $(\mbox{PC})(\tilde{u}_0, s_0, b)$, 
Moreover, to find a solution $(\mbox{PC})(\tilde{u}_0, s_0, b)$ we first consider the following auxiliary problem $(\mbox{AP})(\tilde{u}_0, s, b)$: For a given function $s(t)$ on $[0, T]$, find $\tilde{u}(t, y)$ satisfying: 
\begin{align*}
& \tilde{u}_t(t, y)-\frac{1}{s^2(t)}\tilde{u}_{yy}(t, y)=\frac{ys_t(t)}{s(t)}\tilde{u}_y(t, y) \mbox{ for }(t, y)\in Q(T), \\ %\label{1-7} \\
& -\frac{1}{s(t)}\tilde{u}_y(t, 0)=\beta(b(t)-\gamma \tilde{u}(t, 0)) \mbox{ for }t\in(0, T), \\ % \label{1-8}\\
& -\frac{1}{s(t)}\tilde{u}_y(t, 1)=a_0 \sigma(\tilde{u}(t, 1))(\sigma(\tilde{u}(t, 1))-\alpha s(t)) \mbox{ for }t\in (0, T),  \\ % \label{1-9}\\
& \tilde{u}(0, y)=\tilde{u}_0(y) \mbox{ for }y \in [0, 1]. %\label{1-10}
\end{align*}

Here, we set $g_{\alpha}(r)=a_0 \sigma(r)(\sigma(r)-\alpha s(t))$ for $r\in \mathbb{R}$ and $t\in (0,T)$. In the case $\alpha=0$, due to the function $\sigma$, $g_{0}(r)=a_0(\sigma(r))^2$ is monotonically increasing with respect to $r$. 
%Accordingly, %$(\mbox{AP})(\tilde{u}_0, s, b)$ can be written by a Cauchy problem consisting of a evolution equation associated by the subdifferntial of a convex function and the initial data (\ref{1-10}). %namely, its as the gradient of a convex functions
%namely, $r\to \int_0^r g_0(\xi) d\xi$ is a convex function on $\mathbb{R}$. 
Due to this fact, 
we can use the theory of evolution equations governed by subdifferentials of convex functions (cf. \cite{kenmochi} and references therein) and find a solution $\tilde{u}$ of $(\mbox{AP})(\tilde{u}_0, s, b)$ in the case $\alpha=0$. However, looking now at the case $\alpha>0$, since $g_{\alpha}(r)$ is not monotonic anymore with respect to $r$, we cannot apply the theory of evolution equations directly. Hence, the existence of a solution to $(\mbox{AP})(\tilde{u}_0, s, b)$ in the case $\alpha>0$ is not at all clear. As far as we are aware, the type of free boundary problems is novel. %We will develop within this framework a methodology to handle them in the 1D case. 
We refer the reader to \cite{Visintin} (and related references) for explanations of how and why free-boundary problems can be used to model fast transitions in materials.
%Here, we define a solution in a weak formulation. 

%This is the difficulty in the case $\alpha>0$.

The purpose of this paper is to establish a methodology to deal with the existence of locally-in-time solutions to $(\mbox{PC})(u_0, s_0, b)$ in the case $\alpha >0$. 
For this to happen, we consider  weak solutions to $(\mbox{PC})(\tilde{u}_0, s_0, b)$. The definition of our concept of weak solution is explained in Section 2. As next step, we proceed in the following way:  For a given function $\eta$ on $Q(T)$ and $\varepsilon>0$, we consider in Section 3 a smooth approximation $\eta_{\varepsilon} $ of $\eta$ and construct a solution $\tilde{u}$ to the following auxiliary problem: For given $s$ on $[0, T]$, find $\tilde{u}(t, y)$ such that it holds
\begin{align*}
& \tilde{u}_t(t, y)-\frac{1}{s^2(t)}\tilde{u}_{yy}(t, y)=\frac{ys_t(t)}{s(t)}\tilde{u}_y(t, y) \mbox{ for }(t, y)\in Q(T), \\
& -\frac{1}{s(t)}\tilde{u}_y(t, 0)=\beta(b(t)-\gamma \tilde{u}(t, 0)) \mbox{ for }t\in(0, T), \\
& -\frac{1}{s(t)}\tilde{u}_y(t, 1)=a_0 (\sigma(\tilde{u}(t, 1)))^2-\alpha \sigma(\eta_{\varepsilon}(t)) s(t)) \mbox{ for }t\in (0, T), \\
& \tilde{u}(0, y)=\tilde{u}_0(y) \mbox{ for }y \in [0, 1]. 
\end{align*}
The plan is to obtain uniform estimates of solutions with respect to $\varepsilon$. After that, by the limiting process $\varepsilon \to 0$ and benefitting of Banach's fixed point theorem, we wish to construct a weak solution $\tilde{u}$ of $(\mbox{AP})(\tilde{u}_0, s, b)$. 
In Section 4, we define a solution mapping $\Gamma_T$ between $s$ and the weak solution $\tilde{u}$ of $(\mbox{AP})(\tilde{u}_0, s, b)$. We show that,  for some $T'\leq T$, the mapping $\Gamma_{T'}$ is a contraction mapping on a suitable function space. Finally,  using Banach's fixed point theorem, we prove the existence and uniqueness of a weak solution to the coupled problem $(s, \tilde{u})$ of $(\mbox{PC})(\tilde{u}_0, s_0, b)$ on $[0, T']$. Additionally, we show that the maximal length of the free boundary $s$ is {\em a priori} determined by given parameters, time derivative of the function $b$, and initial data $s_0$. Moreover, we guarantee that the solution $\tilde{u}$ is non-negative and bounded on $Q(T)$. 
%%%%%%%%%%%%%%%%%%%%%%%%%%%%%%%%%%%%%%%%%%%%%%%%%%%%%%%%%%%%%%%%%%%%%%%%%%%%%%%%%%%%%%%%%
%
%
%  // section 2
%
%
%%%%%%%%%%%%%%%%%%%%%%%%%%%%%%%%%%%%%%%%%%%%%%%%%%%%%%%%%%%%%%%%%%%%%%%%%%%%%%%%%%%%%%%%%%
\section{Notation, assumptions and results}
\label{na}

In this paper, we use the following notations. We denote by $| \cdot |_X$ the norm for a Banach space $X$. The norm and the inner product of a Hilbert space $H$ are
 denoted by $|\cdot |_H$ and $( \cdot, \cdot )_H$, respectively. Particularly, for $\Omega \subset {\mathbb R}$, we use the notation of the usual Hilbert spaces $L^2(\Omega)$, $H^1(\Omega)$, and $H^2(\Omega)$.
Throughout this paper, we assume the following parameters and functions:

(A1) $a_0$, $\alpha$, $\gamma$, $\beta$ and $T$ are positive constants.

(A2) $s_0>0$ and $u_0\in L^{\infty}(0, s_0)$ such that $u_0 \geq 0$ on $[0, s_0]$.

(A3) $b\in W^{1,2}(0, T)$ with $b \geq 0$ on $(0, T)$. Also, we set 
\begin{align*}
\displaystyle{b^*=\max \biggl \{\max_{0\leq t\leq T}b(t), r|u_0|_{L^{\infty}(0, s_0)}} \biggr \}.
\end{align*}
%$b(t) \leq b^*$ on $[0, T]$ and $u_0 \leq b^*/\gamma $ on $[0, s_0]$.
%:=\{(t, z) | 0\leq t\leq T, 0<z<s(t)\}$.
%Next, we define our concept of solution to (P)$(u_0, s_0, b)$ on $[0,T]$ in the following way:
%\begin{defin}
%\label{def0}
%For $T>0$, let $s$ be a function on $[0, T]$ and $u$ be a function on $Q_s(T)$. We call that pair $(s, u)$ a solution to (P)$(u_0, s_0, b)$ on $[0, T]$ if the following conditions (S1)-(S6) hold:
%(S1) $s\in W^{1,\infty}(0, T)$, $0<s$ on $[0, T]$, $u\in L^{\infty}(Q_s(T))$, $u_t$, $u_{zz}\in L^2(Q_s(T))$ and $t \in [0, T] \to |u_z(t,  \cdot)|_{L^2(0, s(t))}$ is bounded;
%(S2) $u_t-u_{zz}=0$ on $Q_s(T)$;
%(S3) $-u_z(t, 0)=\beta(b(t)-\gamma u(t, 0))$ for a.e. $t\in [0, T]$;
%(S4) $-u_z(t, s(t))=u(t, s(t))s_t(t)$ for a.e. $t\in [0, T]$;
%(S5) $s_t(t)=a_0 (\sigma(u(t, s(t))-\alpha s(t)) $ for a.e. $t\in [0, T]$;
%(S6) $s(0)=s_0$ and $u(0, z)=u_0(z)$ for $z \in [0, s_0]$.
%\end{defin}

%The first result of this paper is concerned with the existence and uniqueness of a  locally-in-time solution  in the sense of Definition \ref{def0} to the problem (P)$(u_0, s_0, b)$.
%\begin{thm}
%\label{t1}
%Let $T>0$. If (A1)-(A3) hold, then there exists $T^* \leq T$ such that (P)$(u_0, s_0, b)$ has a unique solution $(s, u)$ on $[0, T^*]$ satisfying $s_t(t)\geq 0$ for $t\in [0, T]$ and $0 \leq u \leq u^*$ on $Q_s(T^*)$, where $u^*$ is a positive constant which is independent of $T$. 
%\end{thm}
%To prove Theorem \ref{t1}, 
It is convenient to consider (P)$(u_0, s_0, b)$ transformed into a non-cylindrical domain. Let $T>0$. For given $s\in W^{1,2}(0, T)$ with $s(t)>0$ on $[0, T]$, we introduce the following new function obtained by the change of variables and fix the moving domain:
\begin{align}
\label{2-0}
\tilde{u}(t, y)=u(t, ys(t)) \mbox{ for } (t, y)\in Q(T):=(0, T)\times (0, 1). \tag{2.1} 
\end{align}
By using the function $\tilde{u}$,
(P)$(u_0, s_0, b)$ becomes the following problem (PC)$(\tilde{u}_0, s_0, b)$ posed on the non-cylindrical domain $Q(T)$:
\begin{align}
& \tilde{u}_t(t, y)-\frac{1}{s^2(t)}\tilde{u}_{yy}(t, y)=\frac{ys_t(t)}{s(t)}\tilde{u}_y(t, y) \mbox{ for }(t, y)\in Q(T), \label{2-1} \tag{2.2}\\
& -\frac{1}{s(t)}\tilde{u}_y(t, 0)=\beta(b(t)-\gamma \tilde{u}(t, 0)) \mbox{ for }t\in(0, T), \label{2-2} \tag{2.3}\\
& -\frac{1}{s(t)}\tilde{u}_y(t, 1)=\sigma(\tilde{u}(t, 1))s_t(t) \mbox{ for }t\in (0, T),  \label{2-3} \tag{2.4}\\
& s_t(t)=a_0 (\sigma(\tilde{u}(t, 1))-\alpha s(t)) \mbox{ for }t\in (0, T), \label{2-4} \tag{2.5} \\
& s(0)=s_0, \label{2-5} \tag{2.6} \\
& \tilde{u}(0, y)=u_0(y s(0))(:=\tilde{u}_0(y)) \mbox{ for }y \in [0, 1]. \label{2-6} \tag{2.7}
\end{align}

Here, we introduce the following function space: For $T>0$, we put $H=L^2(0, 1)$, $X=H^1(0, 1)$, $V(T)=L^{\infty}(0, T; H)\cap L^2(0, T; X)$ and $|z|_{V(T)}=|z|_{L^{\infty}(0, T; H)} + |z_y|_{L^2(0, T; H)}$ for $z\in V(T)$. Note that $V(T)$ is a Banach space with the norm $|\cdot|_{V(T)}$. Also, we denote by $X^*$ and $\langle \cdot, \cdot \rangle_{X}$ the dual space of $X$ and the duality pairing between $X$ and $X^*$, respectively.

We define now our concept of solutions to (PC)$(\tilde{u}_0, s_0, b)$ on $[0,T]$ in the following way:
\begin{definition}
\label{def1}
For $T>0$, let $s$ be a function on $[0, T]$ and $\tilde{u}$ be a function on $Q(T)$, respectively. We call that a pair $(s, \tilde{u})$ is a solution of $(\mbox{P})(\tilde{u}_0, s_0, b)$ on $[0, T]$ if the next conditions (S1)-(S4) hold:

(S1) $s \in W^{1, \infty}(0, T)$, $s>0$ on $[0, T]$, $\tilde{u}\in W^{1,2}(0, T; X^*)\cap V(T)$.

(S2) 
\begin{align*}
&\int_0^T \langle \tilde{u}_t(t), z(t) \rangle_X dt + \int_{Q(T)} \frac{1}{s^2(t)} \tilde{u}_{y}(t) z_y(t) dydt + \int_0^T \frac{1}{s(t)}\sigma(\tilde{u}(t, 1))s_t(t) z(t, 1)dt \\
&-\int_0^T \frac{1}{s(t)} \beta(b(t)-\gamma \tilde{u}(t, 0)) z(t, 0) dt = \int_{Q(T)} \frac{ys_t(t)}{s(t)}\tilde{u}_y(t) z(t) dydt
\mbox{ for }z\in V(T). 
\end{align*}

(S3) $s_t(t)=a_0 (\sigma(\tilde{u}(t, 1))-\alpha s(t)) \mbox{ for a.e. }t\in (0, T)$.

(S4) $s(0)=s_0$ and $\tilde{u}(0, y)=\tilde{u}_0(y)$ for a.e. $y\in [0, 1]$.
\end{definition}

The main result of this paper is the existence and uniqueness of a locally-in-time solution of $(\mbox{PC})(\tilde{u}_0, s_0, b)$. We state this result in the next theorem.

\begin{theorem}
\label{t2}
Let $T>0$. If (A1)-(A3) hold, then there exists $T^*\leq T$ such that $(\mbox{PC})(\tilde{u}_0, s_0, b)$ has a unique solution $(s, \tilde{u})$ on $[0, T^*]$. Moreover, the function $\tilde{u}$ is non-negative and bounded on $Q(T)$.
\end{theorem}

%%%%%%%%%%%%%%%%%%%%%%%%%%%%%%%%%%%%%%%%%%%%%%%%%%%%%%%%%%%%%%%%%%%%%%%%%%%%%%%%%%%%%%%%%
%
%
%
% Section 3 // Auxiliary Problem 
%
%
%
%%%%%%%%%%%%%%%%%%%%%%%%%%%%%%%%%%%%%%%%%%%%%%%%%%%%%%%%%%%%%%%%%%%%%%%%%%%%%%%%%%%%%%%%%%%
\section{Auxiliary Problems}
\label{AP}

In this section, %we prove Theorem \ref{t2} on the existence and uniqueness of a locally-in-time  solution of $(\mbox{PC})(\tilde{u}_0, s, b)$. 
%To do so, 
we consider the following auxiliary problem $(\mbox{AP}1)(\tilde{u}_0, s, b)$: For $T>0$ and $s\in W^{1, 2}(0, T)$ with $s>0$ on $[0, T]$
\begin{align}
& \tilde{u}_t(t, y)-\frac{1}{s^2(t)}\tilde{u}_{yy}(t, y)=\frac{ys_t(t)}{s(t)} \tilde{u}_y(t, y) \mbox{ for }(t, y)\in Q(T), \label{3-1} \tag{3.1} \\
& -\frac{1}{s(t)}\tilde{u}_y(t, 0)=\beta(b(t)-\gamma \tilde{u}(t, 0)) \mbox{ for }t\in(0, T), \label{3-2} \tag{3.2} \\
& -\frac{1}{s(t)}\tilde{u}_y(t, 1)=a_0 \sigma(\tilde{u}(t, 1))(\sigma(\tilde{u}(t, 1))-\alpha s(t))\mbox{ for }t\in (0, T), \tag{3.3} \\
%& -\frac{1}{s(t)}\tilde{u}_y(t, 1)=a_0 (\tilde{u}(t, 1)\sigma(\tilde{u}(t, 1))-\alpha (\rho_{\varepsilon}*\tilde{u})(t, 1)s(t))\mbox{ for }t\in (0, T),  \label{3-3}\\
& \tilde{u}(0, y)=\tilde{u}_0(y) \mbox{ for }y \in [0, 1],  \label{3-4} \tag{3.4}
\end{align}
where $\sigma$ is the same function as in (1.6).

\begin{definition}
\label{def2}
For $T>0$, let $\tilde{u}$ be a function on $Q(T)$, respectively. We call that a function $\tilde{u}$ is a solution of $(\mbox{AP})(\tilde{u}_0, s, b)$ on $[0, T]$ if the conditions (S'1)-(S'3) hold:

(S'1) $\tilde{u}\in W^{1,2}(0, T; X^*)\cap V(T)$.

(S'2) 
\begin{align*}
&\int_0^T \langle \tilde{u}_t(t), z(t) \rangle_X dt + \int_{Q(T)} \frac{1}{s^2(t)} \tilde{u}_{y}(t) z_y(t) dydt \\
& +\int_0^T \frac{a_0}{s(t)}\sigma(\tilde{u}(t, 1))(\sigma(\tilde{u}(t, 1))-\alpha s(t))z(t, 1)dt \\
&-\int_0^T \frac{1}{s(t)} \beta(b(t)-\gamma \tilde{u}(t, 0)) z(t, 0) dt = \int_{Q(T)} \frac{ys_t(t)}{s(t)}\tilde{u}_y(t) z(t) dydt
\mbox{ for }z\in V(T). 
\end{align*}

(S'3) $\tilde{u}(0, y)=\tilde{u}_0(y)$ for $y\in [0, 1]$. 
\end{definition}

%To find a solution $\tilde{u}$ of $(\mbox{AP})(\tilde{u}_0, s, b)$, 
Now, we introduce the following problem $(\mbox{AP}2)(\tilde{u}_0, s, \eta, b)$: For $s\in W^{1,\infty}(0, T)$ with $s>0$ on $[0, T]$ and $\eta \in V(T)$ 
%Now, 
%we introduce the following problem $(\mbox{AP})(\tilde{u}_0, s, \eta, b)$: for given $\eta \in V(T)$, 
\begin{align}
& \tilde{u}_t(t, y)-\frac{1}{s^2(t)}\tilde{u}_{yy}(t, y)=\frac{ys_t(t)}{s(t)} \tilde{u}_y(t, y) \mbox{ for }(t, y)\in Q(T), \tag{3.5} \\
& -\frac{1}{s(t)}\tilde{u}_y(t, 0)=\beta(b(t)-\gamma \tilde{u}(t, 0)) \mbox{ for }t\in(0, T), \tag{3.6} \\
%& -\frac{1}{s(t)}\tilde{u}_y(t, 1)=a_0 \tilde{u}(t, 1)(\sigma(\tilde{u}(t, 1))-\alpha s(t))\mbox{ for }t\in (0, T), 
& -\frac{1}{s(t)}\tilde{u}_y(t, 1)=a_0 ((\sigma(\tilde{u}(t, 1))^2-\alpha \sigma(\eta (t, 1))s(t))\mbox{ for }t\in (0, T),  \tag{3.7} \\
& \tilde{u}(0, y)=\tilde{u}_{0}(y) \mbox{ for }y \in [0, 1], \tag{3.8}
\end{align}

%First, for $s_0>0$, $a>0$ such that $a<s_0$, given  $s\in W^{1,\infty}(0, T)$  with $s(0)=s_0$ and $s\geq a$ on $[0, T]$ and $\eta \in V(T)$ we consider $(\mbox{AP})(\tilde{u}_0, s, \eta, b)$ on $[0, T]$. %construct a solution $\tilde{u}$ of $(\mbox{AP})(\tilde{u}_0, s, \eta, b)$ on $[0, T]$. 
The definition of solutions of $(\mbox{AP}2)(\tilde{u}_0, s, \eta, b)$ is Definition \ref{def2} with   (S'2) replaced by (S''2), which now reads
\\
(S''2): 
\begin{align*}
&\int_0^T \langle \tilde{u}_t(t), z(t) \rangle_X dt + \int_{Q(T)} \frac{1}{s^2(t)} \tilde{u}_{y}(t) z_y(t) dydt \\ 
& + \int_0^T \frac{a_0}{s(t)}((\sigma(\tilde{u}(t, 1))^2-\alpha \sigma(\eta(t, 1))s(t))z(t, 1)dt \\
&-\int_0^T \frac{1}{s(t)} \beta(b(t)-\gamma \tilde{u}(t, 0)) z(t, 0) dt = \int_{Q(T)} \frac{ys_t(t)}{s(t)}\tilde{u}_y(t) z(t) dydt
\mbox{ for }z\in V(T). 
\end{align*}
 
First, we construct a solution $\tilde{u}$ of $(\mbox{AP}2)(\tilde{u}_0, s, \eta, b)$ on $[0, T]$. 
To do so, %construct a solution $\tilde{u}$ of $(\mbox{AP})(\tilde{u}_0, s, \eta, b)$,
 for each $\varepsilon>0$ we solve the following problem $(\mbox{AP}2)_{\varepsilon}(\tilde{u}_{0\varepsilon}, s, \eta, b)$: 
\begin{align*}
& \tilde{u}_t(t, y)-\frac{1}{s^2(t)}\tilde{u}_{yy}(t, y)=\frac{ys_t(t)}{s(t)} \tilde{u}_y(t, y) \mbox{ for }(t, y)\in Q(T), \\
& -\frac{1}{s(t)}\tilde{u}_y(t, 0)=\beta(b(t)-\gamma \tilde{u}(t, 0)) \mbox{ for }t\in(0, T), \\
& -\frac{1}{s(t)}\tilde{u}_y(t, 1)=a_0 ((\sigma(\tilde{u}(t, 1))^2 -\alpha \sigma((\rho_{\varepsilon}*\eta)(t, 1))s(t)) \mbox{ for }t\in (0, T), \\
& \tilde{u}(0, y)=\tilde{u}_{0\varepsilon}(y) \mbox{ for }y \in [0, 1].
\end{align*}
Here $\rho_{\varepsilon}$ is a mollifier with support $[-\varepsilon, \varepsilon]$ in time and $\rho_{\varepsilon}*\eta$ is the convolution of $\rho_{\varepsilon}$ with $\eta$:
\begin{align}
\label{3-5}
(\rho_{\varepsilon}*\eta)(t, 1)=\int_{-\infty}^{\infty} \rho_{\varepsilon}(t-s) \overline{\eta}(s, 1) ds \mbox{ for } t\in [0, T], \tag{3.9}
\end{align}
where $\overline{\eta}(t, 1)=\eta(t, 1)$ for $t\in (0, T)$ and vanishes otherwise. 
%is the extension of $\tilde{u}$ to $\mathbb{R}$ given by 
%\[
%\overline{\tilde{u}}(t, 1):=\begin{cases}
%0  & \mbox{ for }t<0, \\
%\tilde{u}(t, 1) & \mbox{ for } 0\leq t \leq T, \\
%0  & \mbox{ for }t>T. 
%\end{cases}
%\]
Also, $\tilde{u}_{0\varepsilon}$ is an approximation function of $\tilde{u}_0$ such that $\{ \tilde{u}_{0\varepsilon}\} \subset X$, $|\tilde{u}_{0\varepsilon}|_H\leq |\tilde{u}_0|_H+1$ and $\tilde{u}_{0\varepsilon}\to \tilde{u}_0$ in $H$ as $\varepsilon \to 0$.

Now, we define a family $\{ \psi^t \}_{t\in[0, T]}$ of time-dependent functionals $\psi^t: H \to \mathbb{R}\cup \{+\infty \}$
for $t\in [0, T]$ as follows: %2/26
%$\psi^t(u):=$
$$
\psi^t(u):=\begin{cases}
             \displaystyle{\frac{1}{2s^2(t)} \int_0^1|u_y(y)|^2 dy} + \displaystyle{\frac{1}{s(t)}\int_0^{u(1)} a_0 ((\sigma(\xi))^2 d\xi -a_0 \alpha u(1)  \sigma((\rho_{\varepsilon}*\eta)(t, 1))}\\
             -\displaystyle{\frac{1}{s(t)}\int_0^{u(0)}\beta(b(t)-\gamma\xi)d\xi} \mbox{ if } u\in D(\psi^t),\\
             +\infty \quad \rm{otherwise},
             \end{cases}
$$
where $D(\psi^t)=X$ for $t\in [0, T]$. Here, we show the property of $\psi^t$.

%%%%%%%%%%%%%%%%%%%%%%%%%%%%%%%%%%%%%%%%%%%%%%%%%%%%%%%%%%%%%%%%%%%%%%%%%%%%%%%%%%%%%%%%
%
%
% Lemma 3.2// properties of psi^t
%
%
%%%%%%%%%%%%%%%%%%%%%%%%%%%%%%%%%%%%%%%%%%%%%%%%%%%%%%%%%%%%%%%%%%%%%%%%%%%%%%%%%%%%%%%%%%
\begin{lemma}
\label{lem1}
Let $s\in W^{1,2}(0, T)$ with $s >0$ on $[0, T]$, $\eta \in V(T)$ and assume (A1)-(A3). Then the following statements hold:
\begin{itemize}
\item[(1)] There exists positive constants $C_0$ and $C_1$ such that the following inequalities hold:
\begin{align*}
(i) & \ |u(y)|^2 \leq C_0 \psi^t(u)+C_1 \mbox{ for }u\in D(\psi^t) \mbox{ and } y=0, 1,\\
%(ii) & \ |u(1)|^2 \leq C_0 \psi^t(u)+C_1 \mbox{ for }u\in D(\psi^t)\\
(ii) & \ \frac{1}{2s^2(t)}|u_y|^2_H \leq C_0 \psi^t(u)+C_1 \mbox{ for }u\in D(\psi^t), 
%(iii) & \ \psi^t(u) +C_2 \geq 0 \mbox{ for }u\in L^2(0, 1).
%(iii) & \ \frac{1}{s(t)} \int_0^{u(0)} a_0 \xi^2 d\xi \leq C_0 \psi^t(u)+C_1 \mbox{ for }u\in D(\psi^t) 
\end{align*}
\item[(2)] For $t\in [0, T]$, the functional $\psi^t$ is proper, lower semi-continuous, and  convex  on $H$.
\end{itemize}
\end{lemma}

\begin{proof}
%First, we note that for $t\in [0, T]$ if $u\in D(\psi^t)$, then $u(0)$ and $u(1)$ are non-negative.
We fix $\varepsilon>0$ and let $t\in [0, T]$, $u\in D(\psi^t)$ and put $l=\displaystyle{\max_{0\leq t \leq T} |s(t)|}$ and $\eta_{\varepsilon}(t)=(\rho_{\varepsilon}*\eta)(t, 1)$ for $t\in [0, T]$. If $u(1)<0$, $\frac{1}{s(t)} \int_0^{u(1)} a_0 (\sigma(\xi))^2 d\xi =0$. If $u(1)\geq 0$, 
then it holds 
\begin{align*}
& \frac{1}{s(t)} \int_0^{u(1)} a_0 (\sigma(\xi))^2 d\xi -a_0 \alpha u(1) \sigma(\eta_{\varepsilon}(t))= \frac{a_0}{s(t)} \frac{1}{3}u^3(1)-a_0 \alpha u(1) \sigma(\eta_{\varepsilon}(t))\\
& \geq  \frac{a_0}{3s(t)} u^3(1) - \frac{1}{3} \delta^{3} u^3(1) - \frac{2}{3\delta^{3/2}} (a_0 \alpha \sigma(\eta_{\varepsilon}(t)))^{\frac{3}{2}}, 
\end{align*}
where $\delta$ is an arbitrary positive number. By using the fact that $\sigma(r)\leq |r|$ for $r\in \mathbb{R}$ and taking a suitable $\delta=\delta_0$ we have 
\begin{align}
\frac{1}{s(t)} \int_0^{u(1)} a_0 (\sigma(\xi))^2 d\xi -a_0 \alpha u(1) \eta_{\varepsilon}(t) \geq -\frac{2}{3\delta^{3/2}_0} (a_0 \alpha \eta^*_{\varepsilon})^{\frac{3}{2}}, \tag{3.10}
\label{3-5-1}
\end{align}
where $\eta^*_{\varepsilon}=\displaystyle{\max_{0\leq t \leq T}|\eta_{\varepsilon}(t)|}$. Moreover, for both cases $u(0)<0$ and $u(0)\geq 0$,  
we observe that 
\begin{align}
 - \frac{1}{s(t)} \int_0^{u(0)} \beta (b(t)-\gamma \xi) d\xi 
%=&\int_{0}^{\varphi(a)} a_0\sigma(\xi)(\sigma(\xi)-\varphi(s(t))) d\xi +\int_{\varphi(a)}^{u(1)} a_0\sigma(\xi)(\sigma(\xi)-\varphi(s(t))) d\xi \\
%=& \int_{0}^{} a_0\varphi(a)(\varphi(a)-\varphi(s(t))) d\xi +\int_{\varphi(a)}^{u(1)} a_0\xi(\xi-\varphi(s(t))) d\xi \\
& =\frac{\beta}{s(t)} \left [ \frac{\gamma}{2} u^2(0) -b(t) u(0)\right] \nonumber \\
\geq \frac{\beta \gamma}{2l} u^2(0)-\frac{\beta b^*}{a} u(0) & \geq \frac{\beta \gamma}{4l} u^2(0)-\frac{\beta l}{\gamma} \left(\frac{b^*}{a} \right)^2. \tag{3.11}
%& +a_0\frac{u^3(1)}{3}-a_0 \varphi(s(t))\frac{u^2(1)}{2}-\biggl(a_0\frac{\varphi^3(a)}{3}-a_0 \varphi(s(t))\frac{\varphi^2(a)}{2}\biggr)\\
%=& a_0\frac{u^3(1)}{3}-a_0 \varphi(s(t))\frac{u^2(1)}{2}+\biggl(a_0\frac{2\varphi^3(a)}{3}-a_0 \varphi(s(t))\frac{\varphi^2(a)}{2}\biggr)\\
%\geq &\frac{a_0}{3} u^3(1) (1-2\eta^{3/2})-\frac{a_0}{3}\left(\frac{c_{\varphi}}{2\eta}\right)^3
%+\biggl(a_0\frac{2\varphi^3(a)}{3}-a_0 \frac{\varphi^2(a)c_{\varphi}}{2}\biggr),
\label{3-5-2}
\end{align}
%Since the second term of the right-hand side of $\psi^t$ is positive, 
Accordingly, if $u(1)<0$, then we have that 
\begin{align}
\label{3-5-3}
\psi^t(u) &\geq \frac{1}{2s^2(t)} \int_0^1|u_y(y)|^2 dy + \frac{\beta \gamma}{4l} u^2(0)-\frac{\beta l}{\gamma} \left(\frac{b^*}{a}\right)^2. 
\tag{3.12}
\end{align}
If $u(1)\geq 0$, then, by (\ref{3-5-1}) and (\ref{3-5-2}) 
we also have that
\begin{align}
\psi^t(u) &\geq \frac{1}{2s^2(t)} \int_0^1|u_y(y)|^2 dy -\frac{2}{3\delta^{3/2}_0} (a_0 \alpha \eta^*_{\varepsilon})^{\frac{3}{2}} + \frac{\beta \gamma}{4l} u^2(0)-\frac{\beta l}{\gamma} \left(\frac{b^*}{a}\right)^2. \tag{3.13}
%\psi^t(u) &\geq \frac{k}{2(s(t)-a)^2} \int_0^1|u_y|^2 dy \\
%& + \frac{c_0\varphi(a)}{2} u^2(1)-c_1 + \frac{c_{\beta}H}{4(L-a)}u^2(0)-\frac{c_{\beta}|h|^2_{L^{\infty}(0, T)}}{H\delta_s}.
\label{3-6-1}
\end{align}
Moreover, it holds that
\begin{align}
|u(1)|^2 & =\biggl | \int_0^1 u_y(y) dy + u(0) \biggr|^2 \leq 2\left (\int_0^1 |u_y(y)|^2 dy + |u(0)|^2 \right) \nonumber \\
& \leq 2\left (\frac{2l^2}{2s^2(t) }\int_0^1 |u_y(y)|^2 dy + |u(0)|^2 \right). \tag{3.14}
\label{3-6-2}
\end{align}
%Moreover, 
%\begin{align}
 %|\psi^t(u)| & \leq \frac{1}{2s^2(t)} \int_0^1|u_y(y)|^2 dy + \frac{1}{s(t)}\int_0^{u(1)} |a_0 \xi (\sigma(\xi) -\alpha s(t))| d\xi\nonumber \\
%& + \frac{1}{s(t)}\int_0^{u(0)}|\beta(b(t)-\gamma\xi)|d\xi
%\label{3-6-2}
%\end{align}
Therefore, by (\ref{3-5-3})-(\ref{3-6-2}) we see that (i) and (ii) of Lemma \ref{lem1} hold.

%Next, we prove (iii). By the definition of $\psi^t$ and $\sigma$ and (\ref{3-5-2}), it is easy to see that 
%\begin{align*}
%\psi^t(u) & \geq -a_0 \alpha u(1) \eta_{\varepsilon}(t)-\frac{1}{s(t)} \int_0^{u(0)} \beta(b(t)-\gamma \xi) d\xi \nonumber \\
%& \geq -a_0 \alpha \left (\frac{\eta^*_{\varepsilon}}{2}+ \frac{u^2(1)}{2}\right) -\frac{\beta l}{\gamma} \left(\frac{b^*}{a} \right)^2.
%\end{align*}
%Therefore, by using (i) of Lemma \ref{lem1} we have 
%\begin{align*}
%\psi^t(u) \geq -\frac{a_0\alpha }{2}(C_0 \psi^t(u) +C_1)-a_0 \alpha \frac{\eta^*_{\varepsilon}}{2}-\frac{\beta l}{\gamma} \left(\frac{b^*}{a} %\right)^2.
%\end{align*}
%Thus, we see that a positive constant $C_2$ such that (iii) holds. 

Next, we prove statement (2).  For $t\in [0, T]$ and $r\in \mathbb{R}$, put 
\begin{align*}
& g_1(s(t), \eta_{\varepsilon}(t), r) =\frac{1}{s(t)}\int_0^r a_0(\sigma(\xi))^2 d\xi - a_0 \alpha r \sigma(\eta_{\varepsilon}(t)), \\
& g_2(s(t), b(t), r)=-\frac{1}{s(t)}\int_0^r \beta(b(t)-\gamma \xi)d\xi.
\end{align*}
Then, by $g_1(s(t), \eta_{\varepsilon}(t), r) = -a_0 \alpha r \sigma(\eta_{\varepsilon}(t))$ for $r\leq 0$, $g_1(s(t), \eta_{\varepsilon}(t), r) $ is a linear decreasing function for $r\leq 0$. Also, by $s(t)>0$ we see that 
\begin{align*}
& \frac{\partial^2}{\partial r^2} g_1(s(t), \eta_{\varepsilon}(t), r) =\frac{2a_0}{s(t)}r>0 \mbox{ for }r>0, \\
& \frac{\partial^2}{\partial r^2}g_2(s(t), b(t), r)= \frac{\beta \gamma}{s(t)}>0 \mbox{ for } r\in \mathbb{R}.
\end{align*}
%Therefore, $r \mapsto a_0r^2$ and $r\mapsto -\beta(b(t)-Hr)$ are convex 
This means that $\psi ^t$ is convex on $H$. By using (i) and (ii) of Lemma \ref{lem1} together with Sobolev's embedding $X \hookrightarrow C([0, 1])$ in one dimensional case, it is easy to prove that the level set of $\psi^t$ is closed in $H$, fact which ensures to the lower semi-continuity of $\psi^t$. Thus, we see that statement (2) holds.
\end{proof}

Lemma \ref{lem1} guarantess the existence of a solution to $(\mbox{AP}2)_{\varepsilon}(\tilde{u}_{0\varepsilon}, s, \eta, b)$.
%%%%%%%%%%%%%%%%%%%%%%%%%%%%%%%%%%%%%%%%%%%%%%%%%%%%%%%%%%%%%%%%%%%%%%%%%%%%%%%%%%%%%%%%%%
%
%
% Lemma 3.3 // 
%
%
%%%%%%%%%%%%%%%%%%%%%%%%%%%%%%%%%%%%%%%%%%%%%%%%%%%%%%%%%%%%%%%%%%%%%%%%%%%%%%%%%%%%%%%%%%%
\begin{lemma}
\label{lem2}
Let $T>0$ and $\varepsilon>0$. If (A1)-(A3) hold, then, for given $s\in W^{1,\infty}(0, T)$ with $s>0$ on $[0, T]$ and $\eta \in V(T)$, 
%$f\in W^{1,2}(Q(T))\cap L^{\infty}(0, T;H^1(0, 1))$, 
the problem $(\mbox{AP}2)_{\varepsilon}(\tilde{u}_{0\varepsilon}, s, \eta, b)$ admits a unique solution $\tilde{u}$ on $[0, T]$ such that $\tilde{u}\in W^{1,2}(0, T;H)\cap L^{\infty}(0, T;X)$. Moreover, the function $t\to \psi^t(\tilde{u}(t))$ is absolutely continuous on $[0, T]$.
\end{lemma}

\begin{proof}
Let  $\varepsilon>0$ arbitrarily fixed. By Lemma \ref{lem1}, for $t\in [0, T]$ $\psi^t$ is a proper lower semi-continuous convex function on $H$.
From the definition of the subdifferential of $\psi^t$, for $t\in [0, T]$, $z^*\in \partial \psi^t(u)$ is characterized by
$u$, $z^*\in H$, %With this information at hand, we see that $z^*=\partial \psi^t(u)$ if and only if $z^*\in L^2(0, 1)$ and
\begin{align*}
& z^*=-\frac{1}{s^2(t)} u_{yy} \mbox{ on } (0, 1), \\
& -\frac{1}{s(t)}u_y(0)=\beta(b(t)-\gamma u(0)), \ -\frac{1}{s(t)}u_y(1)=a_0 ((\sigma(u(1))^2-\alpha \sigma((\rho_{\varepsilon}*\eta)(t, 1))s(t)).
\end{align*}
Namely, $\partial \psi^t$ is single-valued. Also, we see that there exists a positive constant $C_3$ such that for each $t_1$, $t_2\in [0, T]$ with $t_1\leq t_2$, and for any $u\in D(\psi^{t_1})$, there exists $\bar{u}\in D(\psi^{t_2})$ satisfying the inequality
\begin{align}
\label{3-8}
%&|\bar{u}-u|_{L^2(0, 1)} \leq \alpha |s(t_1)-s(t_2)|(1+|\varphi^{t_1}(u)|^{1/2}), \label{3-7}\\
&|\psi^{t_2}(\bar{u})-\psi^{t_1}(u)| \nonumber \\
\leq & C_3(|s(t_1)-s(t_2)|+|b(t_1)-b(t_2)|+|(\rho_{\varepsilon}*\eta)(t_1, 1)-(\rho_{\varepsilon}*\eta)(t_2, 1)|)(1+|\psi^{t_1}(u)|). 
\tag{3.15}
\end{align}
Indeed, by taking $\bar{u}:=u$ it easy to see that $\bar{u}\in D(\psi^{t_2})$, and using Lemma \ref{lem1}, we can find $C_3>0$ such that (\ref{3-8}) holds. 
%Also, by the definition of $\psi^t$, it holds
%\begin{align*}
%&\psi^{t_2}(\bar{u})-\psi^{t_1}(u) \\
%=&\left(\frac{1}{s(t_2)}-\frac{1}{s(t_1)}\right) \int_0^1 |u_y(y)|^2 dy + \left (\frac{1}{s(t_2)}-\frac{1}{s(t_1)} \right) \int_0^{u(0)} a_0\xi^2 d\xi \\
%& - \left(\frac{1}{s(t_2)}-\frac{1}{s(t_1)}\right) \int_0^{u(1)} \beta(b(t)-H\xi) d\xi -\frac{\beta}{s(t_1)} (b(t_2)-b(t_1))u(0) 
%\end{align*}
%Here, we note that for $t\in [0, T]$, 
%\begin{align*}
%\frac{1}{s(t)}\int_0^{u(1)} a_0 \xi^2 d\xi & = \psi^t(u)- \frac{1}{2s^2(t)} \int_0^1|u_y(y)|^2 dy + \frac{1}{s(t)}\int_0^{u(0)}%\beta(b(t)-H\xi)d\xi \\
%& \leq \psi^t(u) + \frac{\beta b^*}{s(t)}u(0). 
%\leq & \psi^t(u) + \frac{\beta b^*}{2a}(1+u^2(0)).  
%\end{align*}
%Then, by (i) of Lemma \ref{lem1} we see that there exists positive constants $C_2$ and $C_3$ such that 
%\begin{align}
%\label{3-8-1}
% \frac{1}{s(t)}\int_0^{u(1)} a_0 \xi^2 d\xi \leq C_2\psi^t(u) + C_3 \mbox{ for } u\in D(\psi^t). 
%\end{align}
%By the statement (1) of Lemma \ref{lem1} and (\ref{3-8-1}) we infer that (\ref{3-8}) holds. 

Now, for given $f\in L^2(0, T; X)$ we consider the following problem which is denoted by $(\mbox{AP}3)_{\varepsilon}(\tilde{u}_{0\varepsilon}, s, f, \eta, b)$:
\begin{align*}
& \tilde{u}_t(t, y)-\frac{1}{s^2(t)}\tilde{u}_{yy}(t, y)=\frac{ys_t(t)}{s(t)} f_y(t, y) \mbox{ for }(t, y)\in Q(T), \\
& -\frac{1}{s(t)}\tilde{u}_y(t, 0)=\beta(b(t)-\gamma \tilde{u}(t, 0)) \mbox{ for }t\in(0, T), \\
& -\frac{1}{s(t)}\tilde{u}_y(t, 1)=a_0 ((\sigma(\tilde{u}(t, 1))^2 -\alpha \sigma((\rho_{\varepsilon}*\eta)(t, 1))s(t)) \mbox{ for }t\in (0, T), \\
& \tilde{u}(0, y)=\tilde{u}_{0\varepsilon}(y) \mbox{ for }y \in [0, 1].
\end{align*}
Then,  $(\mbox{AP}3)_{\varepsilon}(\tilde{u}_{0\varepsilon}, s, f, \eta, b)$ can be written into the following Cauchy problem (CP)$_{\varepsilon}$: 
\begin{align*}
& \tilde{u}_t+\partial \psi^t(\tilde{u}(t)) =\frac{y s_t(t)}{s(t)}f_y(t) \mbox{ in }H, \\
& \tilde{u}(0, y)=\tilde{u}_{0\varepsilon}(y) \mbox{ for }y\in [0, 1].
\end{align*}
%Here $f\in L^2(0, T;H^1(0, 1))$ and $s \in W^{1,2}(0, T)$ so that 
Since $\frac{ys_t }{s}f_y \in L^2(0, T; H)$, by the general theory of evolution equations governed by time dependent subdifferentials(cf. \cite{kenmochi}) we see that 
(CP)$_{\varepsilon}$ has a solution $\tilde{u}$ on $[0, T]$ such that $\tilde{u}\in W^{1,2}(Q(T))$, $\psi^t(\tilde{u}(t)) \in L^{\infty}(0, T)$ and $t \to \psi^t(\tilde{u}(t))$ is absolutely continuous on $[0, T]$. This implies that $\tilde{u}$ is a unique solution of $(\mbox{AP}3)_{\varepsilon}(\tilde{u}_{0\varepsilon}, s, f, \eta, b)$ on $[0, T]$.
%\end{proof}

Next, for given $f\in V(T)$ we define a operator $\Gamma_T(f)=\tilde{u}$, where
$\tilde{u}$ is a unique solution of $(\mbox{AP}3)_{\varepsilon}(\tilde{u}_{0\varepsilon}, s, f, \eta, b)$.
Now, for $i=1,2$ we put $\Gamma_T(f_i)=\tilde{u}_i$ and $f=f_1-f_2$ and $\tilde{u}=\tilde{u}_1-\tilde{u}_2$.
Then, it holds that 
\begin{align}
& \frac{1}{2}\frac{d}{dt}|\tilde{u}(t)|^2_{H} - \int_0^1 \frac{1}{s^2(t)}\tilde{u}_{yy}(t) \tilde{u}(t) dy = \int_0^1 \frac{y s_t(t)}{s(t)}f_{y}(t)\tilde{u}(t) dy.
\label{3-9} \tag{3.16}
\end{align}
Using the boundary condition, the second term of the left-hand side of (\ref{3-9}) can be as follows:
\begin{align}
& - \int_0^1 \frac{1}{s^2(t)}\tilde{u}_{yy}(t) \tilde{u}(t) dy \nonumber \\
& = \frac{1}{s^2(t)}\left(-\tilde{u}_y(t, 1)\tilde{u}(t, 1)+\tilde{u}_y(t, 0)\tilde{u}(t, 0) + \int_0^1|\tilde{u}_y(t)|^2 dy\right) \nonumber \\
& =\frac{a_0}{s(t)}\biggl((\sigma(\tilde{u}_1(t, 1))^2-\alpha \sigma(\eta_{\varepsilon}(t)) s(t))-((\sigma(\tilde{u}_2(t, 1))^2-\alpha \sigma(\eta_{\varepsilon}(t)) s(t)) \biggr)\tilde{u}(t, 1) \nonumber \\
& -\frac{1}{s(t)}\biggl(\beta(b(t)-\gamma \tilde{u}_1(t, 0))-\beta(b(t)-\gamma \tilde{u}_2(t, 0))\biggr)\tilde{u}(t, 0) + \frac{1}{s^2(t)}\int_0^1|\tilde{u}_y(t)|^2 dy, 
\label{3-10} \tag{3.17}
%& \eq  - \frac{\beta H}{s(t)}|\tilde{u}(t, 0)|^2 + \frac{k}{s^2(t)}\int_0^1|\tilde{u}_y(t)|^2 dy.
\end{align}
where $\eta_{\varepsilon}(t)=(\rho_{\varepsilon}*\eta)(t, 1)$ for $t\in [0, T]$. 
%For the first term of the right-hand side of (\ref{3-10}), we separate as follows:
%\begin{align*}
%& \frac{a_0}{s(t)}\biggl(\tilde{u}_1(t, 1)\sigma(\tilde{u}_1(t, 1))-\tilde{u}_2(t, 1)\sigma(\tilde{u}_2(t, 1))\biggr)\tilde{u}(t, 1) \\
%= & \frac{a_0}{s(t)} \sigma(\tilde{u}_1(t, 1)) \tilde{u}^2(t, 1) 
%+ \frac{a_0}{s(t)} \tilde{u}_2(t, 1)(\sigma(\tilde{u}_1(t, 1)) -\sigma(\tilde{u}_2(t, 1)))\tilde{u}(t, 1) 
%\end{align*}
The first term of the right-hand side of (\ref{3-10}) is non-negative due to the monotonicity and non-negativity of $\sigma$.
%\begin{align}
%\frac{a_0}{s(t)}(\tilde{u}_1(t, 1)(\sigma(\tilde{u}_1(t, 1))-\tilde{u}_2(t, 1)\sigma(\tilde{u}_2(t, 1)))\tilde{u}(t, 1) \geq -\frac{a_0}{s(t)}|\tilde{u}_2(t, 1)||\tilde{u}(t, 1)|^2.
%\label{3-11} 
%\end{align}
Since the second term from the right-hand side of (\ref{3-10}) is non-negative,  we have 
%\begin{align}
%& - \int_0^1 \frac{1}{s^2(t)}\tilde{u}_{yy}(t) \tilde{u}(t) dy \\
%\geq \frac{1}{s^2(t)}\int_0^1|\tilde{u}_y(t)|^2 dy -\frac{a_0}{s(t)}|\tilde{u}_2(t, 1)||\tilde{u}(t, 1)|^2 
%\label{3-11}
%\end{align}
%Accordingly, by (\ref{3-9})-(\ref{3-11}), we have that 
\begin{align}
\label{3-12}
\frac{1}{2}\frac{d}{dt}|\tilde{u}(t)|^2_{H}  + \frac{1}{s^2(t)}\int_0^1|\tilde{u}_y(t)|^2 dy \leq \int_0^1 \frac{ys_t(t)}{s(t)}f_{y}(t) \tilde{u}(t)dy. 
\tag{3.18}
\end{align}
Using H\"older's inequality, it holds that 
\begin{align}
\label{3-13} 
\int_0^1 \frac{y s_t(t)}{s(t)}f_{y}(t) \tilde{u}(t)dy \leq \frac{|s_t|_{L^{\infty}(0, T)}}{a}|\tilde{u}(t)|_{H} |f_y(t)|_{H}. 
\tag{3.19}
\end{align}
%Here, by Sobolev's embedding theorem in one dimension, we note that it holds that 
%\begin{align}
%\label{1d}
%|z(t, y)|^2 \leq C_e |z(t)|_{X} |z(t)|_{H} \mbox{ for } z\in X \mbox{ and } y\in [0, 1], 
%\end{align}
%where $C_e$ is a positive constant defined from Sobolev's embedding theorem. 
%Also, by $\psi^t(\tilde{u}_i(t)) \in L^{\infty}(0, T)$ for $i=1$,2,  Lemma \ref{lem1} and (\ref{1d}) we note that $\tilde{u}_i(\cdot, 1)\in L^{\infty}(0, T)$ for $i=1$, 2. Then, by (\ref{1d}) we have 
%\begin{align}
%\label{3-14}
%\frac{a_0}{s(t)}|\tilde{u}_2(t, 1)||\tilde{u}(t, 1)|^2 
%\leq \frac{a_0}{a}C_e|\tilde{u}_2(\cdot, 1)|_{L^{\infty}(0, T)}(|\tilde{u}_y(t)|_{H} |\tilde{u}(t)|_{H} + |\tilde{u}(t)|^2_H)
%\leq & \frac{1}{2s^2(t)} |\tilde{u}_y(t)|^2_{H} + ((a_0C_e|\tilde{u}_2(t, 1)|)^2 + a_0 C_e|\tilde{u}_2(t, 1)|)|\tilde{u}(t)|^2_H)
%\end{align}

Let $T_1\in (0, T]$. Then, by putting $l=\displaystyle{\mbox{max}_{0\leq t\leq T}|s(t)|}$, integrating (\ref{3-12}) with (\ref{3-13}) over $[0, t]$ for any $t\in [0, T_1]$, we obtain that 
\begin{align}
\label{3-14-1} 
& \frac{1}{2}|\tilde{u}(t)|^2_{H}  + \frac{1}{l^2} \int_0^{t} \int_0^1|\tilde{u}_y(\tau)|^2 dy d\tau \nonumber \\
\leq & \frac{|s_t|_{L^{\infty}(0, T)}}{a} |\tilde{u}|_{L^{\infty}(0, T_1;H)} T^{1/2}_1 \left( \int_0^{T_1} |f_y(\tau)|^2_{H} d\tau \right)^{1/2}\nonumber \\
%& + \frac{a_0C_e}{a}|\tilde{u}_2(\cdot, 1)|_{L^{\infty}(0, T)}|\tilde{u}|_{L^{\infty}(0, T_1;H)} T^{1/2}_1\biggl( \int_0^{T_1}|\tilde{u}_y(t)|^2_H \biggr)^{1/2} \nonumber \\
%& +  \frac{a_0C_e}{a}|\tilde{u}_2(\cdot, 1)|_{L^{\infty}(0, T)}T_1 |\tilde{u}|^2_{L^{\infty}(0, T_1;H)} \nonumber \\
\leq & \frac{|s_t|_{L^{\infty}(0, T)}}{a} |\tilde{u}|_{V(T_1)} T^{1/2}_1 |f|_{V(T_1)}. 
%+ \frac{a_0}{a} C_e |\tilde{u}_2(\cdot, 1)|_{L^{\infty}(0, T)}T^{1/2}_1(1+T^{1/2})|\tilde{u}|^2_{V(T_1)}
\tag{3.20}
\end{align}
Hence, by putting $\delta=\mbox{min}\{1/2, 1/l^2\}$ %and $C(T)=(a_0C_e |\tilde{u}_2(\cdot, 1)|_{L^{\infty}(0, T)}(1+T^{1/2})/a$ 
we have that 
\begin{align*}
& \delta |\tilde{u}|_{V(T_1)} \leq \frac{|s_t|_{L^{\infty}(0, T)}}{a} T^{1/2}_1 |f|_{V(T_1)} \mbox{ for }T_1\in (0, T].
\end{align*}
%Now, we take $T^*_1$ such that $\delta -C(T)(T^*_1)^{1/2}=\frac{\delta}{2}$ and see that $T_1\leq T^*_1$ such that $\frac{2}{\delta} \frac{|s_t|_{L^{\infty}(0, T)}}{a} T^{1/2}_1 <1$. This implies that $\Gamma_{T_1}$ is a contraction mapping in $V(T_1)$. 
From this result, we infer that for some $T_1\leq T$ the mapping $\Gamma_{T_1}$ is a contraction on $V(T_1)$. 
Therefore, by Banach's fixed point theorem, it exists $\tilde{u}\in V(T_1)$ such that
$\Gamma_{T_1}(\tilde{u})=\tilde{u}$ which implies that $\tilde{u}$ is a solution of $(\mbox{AP}2)_{\varepsilon}(\tilde{u}_{0\varepsilon}, s, \eta, b)$ on $[0, T_1]$. Since $T_1$ is independent of the choice of initial data of $\tilde{u}$, by repeating the argument behind the local existence result, we can extend the solution $\tilde{u}$ beyond $T_1$. Thus, the conclusion of Lemma \ref{lem2} holds. 
\end{proof}

As next step, we provide an uniform estimate  with respect to $\varepsilon$ on a solution $\tilde{u}$ of $(\mbox{AP}2)_{\varepsilon}(\tilde{u}_{0\varepsilon}, s, \eta, b)$. 
%Next,  for given $s\in W^{1,2}(0, T)$ with $s(0)=s_0$ and $s \geq a$ on $[0, T]$, we construct a solution to $(\mbox{AP})_{\varepsilon}(\tilde{u}_{0\varepsilon}, s, \eta, b)$.

%%%%%%%%%%%%%%%%%%%%%%%%%%%%%%%%%%%%%%%%%%%%%%%%%%%%%%%%%%%%%%%%%%%%%%%%%%%%%%%%%%%%%%%%%%%%%%%%%%
%
%
% // Lemma 3.4 
%
%
%%%%%%%%%%%%%%%%%%%%%%%%%%%%%%%%%%%%%%%%%%%%%%%%%%%%%%%%%%%%%%%%%%%%%%%%%%%%%%%%%%%%%%%%%%%%%%%%%
\begin{lemma}
\label{lem6}
Let $T>0$, $s\in W^{1,\infty}(0, T)$ with $s>0$ on $[0, T]$, $\eta\in V(T)$ and $\tilde{u}_{\varepsilon}$ be a solution of $(\mbox{AP}2)_{\varepsilon}(\tilde{u}_{0\varepsilon}, s, \eta, b)$ on $[0, T]$ for each $\varepsilon>0$. Then, it holds that 
\begin{align}
%\label{3-15-5}
|\tilde{u}_{\varepsilon}(t)|^2_{H} + \int_0^t |\tilde{u}_{\varepsilon y}(\tau)|^2_{H} dy \leq M(1+ |\eta|^2_{V(T)}) %\int_0^t |(\rho_{\varepsilon}*\eta)(t, 1)|^2 dt)
%\int_0^t (|\tilde{u}_{ny}(\tau)|_{H}|\tilde{u}_n(\tau)|_{H} + |\tilde{u}_n(\tau)|^2_{H})d\tau 
\mbox{ for }t\in[0, T] \mbox{ and }\varepsilon \in (0, 1],
\tag{3.21}
\end{align}
where $M=M(a_0, a, \beta, b^*, T)$ is a positive constant which is independent of $\varepsilon$ and depends on $a_0$, $a$, $\beta$, $b^*$ and $T$, $\displaystyle{a=\min_{0\leq t\leq T}s(t)}$.
%If (A1)-(A3), then, for each $\varepsilon>0$, $(\mbox{AP})_{\varepsilon}(\tilde{u}_0, s, b)$ has a unique solution $\tilde{u}$ on $[0, T]$. 
\end{lemma}

\begin{proof}
Let $\tilde{u}_{\varepsilon}$ be a solution of $(\mbox{AP}2)_{\varepsilon}(\tilde{u}_{0\varepsilon}, s, \eta, b)$ on $[0, T]$ for each $\varepsilon>0$. 
%For given $s\in W^{1, 2}(0, T)$ with $s(0)=s_0$ and $s\geq a$ on $[0, T]$, we choose a sequence $\{s_n\} \subset W^{1, \infty}(0, T)$ and $l>0$ satisfying $a \leq s_n\leq l$ on $[0, T]$ for each $n\in \mathbb{N}$,
%$s_n \to s$ in $W^{1,2}(0, T)$ as $n \to \infty$. Let $\varepsilon>0$. By Lemma \ref{lem5} we can take a sequence $\{\tilde{u}_n\}=\{\tilde{u}_{\varepsilon, n} \}$ of solutions to $(\mbox{AP})_{\varepsilon}(\tilde{u}_0, s_n, \eta, b)$ on $[0, T]$. Then, we see that $t \to \psi^t(\tilde{u}_n(t))$ is absolutely continuous on $[0, T]$ so that $t \to \frac{1}{s^2_n(t)}|\tilde{u}_{ny}(t)|^2_{L^2(0, 1)}$ is continuous on $[0, T]$.
First, it holds that 
\begin{align}
\label{3-15}
& \frac{1}{2}\frac{d}{dt}|\tilde{u}_{\varepsilon}(t)|^2_{H} -\int_0^1 \frac{1}{s^2(t)}\tilde{u}_{\varepsilon yy}(t) \tilde{u}_{\varepsilon}(t) dy = \int_0^1 \frac{y s_{t}(t)}{s(t)} \tilde{u}_{\varepsilon y}(t) \tilde{u}_\varepsilon(t) dy.
\tag{3.22}
\end{align}
The second term on the left-hand side is as follows:
\begin{align*}
& -\int_0^1 \frac{1}{s^2(t)}\tilde{u}_{\varepsilon yy}(t) \tilde{u}_{\varepsilon} (t) dy \\
=& \frac{a_0}{s(t)}((\sigma(\tilde{u}_{\varepsilon}(t, 1))^2-\alpha \sigma(\eta_{\varepsilon}(t)) s(t)) \tilde{u}_{\varepsilon}(t, 1) \\
& -\frac{1}{s(t)}\beta(b(t)-\gamma \tilde{u}_{\varepsilon}(t, 0))\tilde{u}_{\varepsilon}(t, 0) + \frac{1}{s^2(t)}\int_0^1|\tilde{u}_{\varepsilon y}(t)|^2 dy \\
\geq & -a_0\alpha |\eta_{\varepsilon}(t)||\tilde{u}_{\varepsilon}(t, 1)|-\frac{1}{s(t)}\beta(b(t)-\gamma \tilde{u}_{\varepsilon}(t, 0))\tilde{u}_{\varepsilon}(t, 0) + \frac{1}{s^2(t)}\int_0^1|\tilde{u}_{\varepsilon y}(t)|^2 dy, 
\end{align*}
where $\eta_{\varepsilon}(t)=(\rho_{\varepsilon}*\eta)(t, 1)$. 
%By $b\varphi(a) \geq \varphi(s_n(t))$ for $t\in [0, T]$, we note that the first term in the right hand side is positive.
From the above, we obtain that
\begin{align}
& \frac{1}{2}\frac{d}{dt}|\tilde{u}_{\varepsilon}(t)|^2_{H} + \frac{1}{s^2_n(t)}\int_0^1|\tilde{u}_{\varepsilon y}(t)|^2 dy \nonumber \\
\leq &  \int_0^1 \frac{y s_{t}(t)}{s(t)} \tilde{u}_{\varepsilon y}(t) \tilde{u}_{\varepsilon}(t) dy \nonumber \\
& +a_0\alpha |\eta_{\varepsilon}(t)||\tilde{u}_{\varepsilon}(t, 1)| + \frac{1}{s(t)}\beta(b(t)-\gamma \tilde{u}_{\varepsilon}(t, 0))\tilde{u}_{\varepsilon}(t, 0) \mbox{ for } t\in [0, T].
\label{3-15-1} \tag{3.23}
\end{align}
We estimate the right-hand side of (\ref{3-15-1}). First, by Young's inequality we have that 
\begin{align}
\label{3-15-2}
\int_0^1 \frac{y s_{t}(t)}{s(t)} \tilde{u}_{\varepsilon y}(t) \tilde{u}_{\varepsilon}(t) dy \leq \frac{1}{4s^2(t)}\int_0^1|\tilde{u}_{\varepsilon y}(t)|^2dy + |s_{t}(t)|^2 \int_0^1|\tilde{u}_{\varepsilon}(t)|^2dy.
\tag{3.24}
\end{align}
Here, by Sobolev's embedding theorem in one dimension, we note that it holds that 
\begin{align}
\label{1d}
|z(y)|^2 \leq C_e |z|_{X} |z|_{H} \mbox{ for } z\in X \mbox{ and } y\in [0, 1], 
\tag{3.25}
\end{align}
where $C_e$ is a positive constant defined from Sobolev's embedding theorem. 
By (\ref{1d}) and $s \geq a$ on $[0, T]$,  we obtain  
\begin{align}
\label{3-15-3}
& \frac{1}{s(t) }\beta(b(t)-\gamma \tilde{u}_{\varepsilon}(t, 0))\tilde{u}_{\varepsilon}(t, 0) \leq \frac{\beta b^*}{s(t)}|\tilde{u}_{\varepsilon}(t, 0)| \nonumber \\
\leq & \frac{\beta b^*C_e}{2s(t)}\biggl(|\tilde{u}_{\varepsilon y}(t)|_{H}|\tilde{u}_{\varepsilon}(t)|_{H} + |\tilde{u}_{\varepsilon}(t)|^2_{H}\biggr) +
\frac{\beta b^*}{2 s(t)} \nonumber \\
\leq & \frac{1}{4s^2(t)}|\tilde{u}_{\varepsilon y}(t)|^2_{H} + \left( \frac{(\beta b^* C_e)^2}{4}+ \frac{\beta b^* C_e}{2a}\right)|\tilde{u}_{\varepsilon}(t)|^2_{H} + \frac{\beta b^*}{2a},
\tag{3.26}
\end{align}
and %Also, it holds that $|v_{\varepsilon, n}(t)|\leq C_e|\tilde{u}_n(t)|^{1/2}_{H^1(0, 1)}||\tilde{u}_n(t)|^{1/2}_{L^2(0, 1)}$ so that 
\begin{align}
\label{3-15-30}
& a_0\alpha |\eta_{\varepsilon}(t)||\tilde{u}_{\varepsilon}(t, 1)| \leq \frac{a_0\alpha }{2}\biggl ( C_e (|\tilde{u}_{\varepsilon y}(t)|_{H}|\tilde{u}_{\varepsilon}(t)|_{H} + |\tilde{u}_{\varepsilon}(t)|^2_{H} )+ |\eta_{\varepsilon}(t)|^2 \biggr) \nonumber \\
& \leq \frac{1}{4s^2(t)}|\tilde{u}_{\varepsilon y}(t)|^2_{H} + \biggl (s^2(t) \left(\frac{a_0\alpha}{2} C_e\right)^2 + \frac{a_0\alpha}{2} C_e\biggr)|\tilde{u}_{\varepsilon}(t)|^2_{H} + \frac{a_0 \alpha}{2}|\eta_{\varepsilon}(t)|^2.
\tag{3.27}
\end{align}
From (\ref{3-15})-(\ref{3-15-30}), we have that 
\begin{align}
\label{3-15-4}
& \frac{1}{2}\frac{d}{dt}|\tilde{u}_{\varepsilon}(t)|^2_{H} + \frac{1}{4s^2(t)}\int_0^1|\tilde{u}_{\varepsilon y}(\tau)|^2 dy \nonumber \\
\leq & \left( |s_{t}(t)|^2 + \frac{(\beta b^* C_e)^2}{4}+ \frac{\beta b^* C_e}{2a} \right)|\tilde{u}_{\varepsilon}(t)|^2_{H} + \frac{\beta b^*}{2a} \nonumber \\
& + \biggl (s^2(t) \left(\frac{a_0\alpha}{2} C_e\right)^2 + \frac{a_0\alpha}{2} C_e \biggr)|\tilde{u}_{\varepsilon}(t)|^2_{H} + \frac{a_0 \alpha}{2}|\eta_{\varepsilon}(t)|^2
 \mbox{ for }t\in[0, T].
\tag{3.28}
%\tag{3.18}\left( \frac{|s_{nt}(t)|^2}{{2}} + (\beta b^* C_e)^2+ \frac{\beta b^* C_e}{2s_0} \right)+ \biggl (s^2_n(t) \left(\frac{a_0\alpha}{2} C_e\right)^2 + \frac{a_0\alpha}{2} \biggr)
\end{align}
Now, we denote $F(t)$ the coefficient of $|\tilde{u}_{\varepsilon}|^2_{H}$ in the right-hand side. As $s\in W^{1,2}(0, T)$, we observe that $F \in L^1(0, T)$. 
Making use of Gronwall's inequality, we are led to 
\begin{align}
\label{3-15-5}
& \frac{1}{2} |\tilde{u}_{\varepsilon}(t)|^2_{H} + \frac{1}{4l^2}\int_0^t |\tilde{u}_{\varepsilon y}(\tau)|^2_{H} dy \nonumber \\
\leq & \biggl( \frac{1}{2}\biggl (|\tilde{u}_0|^2_{H} + 1+ \frac{\beta b^*}{a}T + a_0\alpha \int_0^{t_1} |\eta_{\varepsilon}(t)|^2 dt\biggr) \int_0^{t_1} F(t) dt \biggr) e^{\int_0^{t_1} F(t) dt} \mbox{ for }t\in[0, T].
\tag{3.29}
%\frac{1}{2} (|\tilde{u}_0|^2_{H}+1) + \int_0^t F(\tau)|\tilde{u}_{\varepsilon}(\tau)|^2_{H} d\tau + \frac{\beta b^*}{2a}T + \frac{a_0 \alpha}{2} \int_0^t |\eta_{\varepsilon}(\tau)|^2 d\tau
%\int_0^t (|\tilde{u}_{ny}(\tau)|_{H}|\tilde{u}_n(\tau)|_{H} + |\tilde{u}_n(\tau)|^2_{H})d\tau 
%\mbox{ for }t\in[0, T].
\end{align}
%Put $I_{\varepsilon}(t)=\int_0^t F(\tau) |\tilde{u}_{\varepsilon}(\tau)|^2_{H} d\tau$ for $t\in [0, T]$. Then, 
%we have 
%\begin{align*}
%\label{3-15-9}
%I'_{\varepsilon}(t) \leq F(t)(|\tilde{u}_0|^2_{H} +1 + \frac{\beta b^*}{a}T + a_0\alpha \int_0^{t} |\eta_{\varepsilon}(\tau)|^2 d\tau) + 2F(t) I_{\varepsilon}(t) \mbox{ for }t\in [0, T].
%\end{align*}
%Accordingly, by Gronwall's inequality we have 
%\begin{align}
%\label{3-15-10}
%I_{\varepsilon}(t_1) \leq \biggl( \biggl (|\tilde{u}_0|^2_{H} + 1+ \frac{\beta b^*}{a}T + a_0\alpha \int_0^{t_1} |\eta_{\varepsilon}(t)|^2 dt\biggr) %\int_0^{t_1} F(t) dt \biggr) e^{\int_0^{t_1} 2F(t) dt}
%\mbox{ for }t_1\in [0, T].
%\end{align}
We note that it holds that 
\begin{align}
\label{3-15-11}
&\int_0^T |\eta_{\varepsilon}(\tau, 1)|^2 d\tau \leq \int_0^T |\eta(\tau, 1)|^2 d\tau \leq C_e \int_0^T (|\eta_y(\tau)|_{H}|\eta(\tau)|_{H} + |\eta(\tau)|^2_H) dt \nonumber \\
& \leq C_e \biggl(|\eta|_{L^{\infty}(0, T; H)} T^{1/2} \biggl( \int_0^T |\eta_y(\tau)|^2_H d\tau \biggr)^{1/2} + T|\eta|^2_{L^{\infty}(0, T; H)}  \biggr)  \nonumber \\
& \leq C_eT^{1/2}(1+T^{1/2}) |\eta|^2_{V(T_1)}.
\tag{3.30}
\end{align}
Therefore, by (\ref{3-15-5}) and (\ref{3-15-11}) we see that there exists a positive constant $M=M(a_0, a, \beta, b^*, T)$  such that Lemma \ref{lem6} holds. 
%By putting the right-hand side of (\ref{3-15-10}) by $M_1(T)$, from (\ref{3-15-8}) we see that 
%\begin{align}
%\label{3-15-11}
%\frac{1}{2} |\tilde{u}_{\varepsilon}(t)|^2_{H} + \frac{1}{4l^2}\int_0^t |\tilde{u}_{\varepsilon y}(\tau)|^2_{H} dy \leq \frac{1}{2} |\tilde{u}_0|^2_{H} + %M_1(T) + \frac{\beta b^*}{2a}T +\frac{a_0\alpha}{2} M(T)\mbox{ for }t\in[0, T].
%\end{align}
%We put the right-hand side of (\ref{3-15-11}) by $M_2$. 
\end{proof}

%%%%%%%%%%%%%%%%%%%%%%%%%%%%%%%%%%%%%%%%%%%%%%%%%%%%%%%%%%%%%%%%%%%%%%%%%%%%%%%%%%%%%%%%%%%%%
%
%
% Lemma 3.5 // 
%
%
%%%%%%%%%%%%%%%%%%%%%%%%%%%%%%%%%%%%%%%%%%%%%%%%%%%%%%%%%%%%%%%%%%%%%%%%%%%%%%%%%%%%%%%%%%%%%

\begin{lemma}
\label{lem7}
Let $T>0$, $s\in W^{1,\infty}(0, T)$ with $s>0$ on $[0, T]$ and $\eta \in V(T)$. If (A1)-(A3), then, $(\mbox{AP}2)(\tilde{u}_0, s, \eta, b)$ has a unique solution $\tilde{u}$ on $[0, T]$. 
\end{lemma}

\begin{proof}
%At the end of this section, by letting $\varepsilon \to 0$ we show the existence of a solution $\tilde{u}$ of $(\mbox{AP})(\tilde{u}_0, s, \eta, b)$ on $[0, T]$. 
Let $s\in W^{1,\infty}(0, T)$ with $s>0$ on $[0, T]$. Then, we already have a solution $\tilde{u}_{\varepsilon}$ of $(\mbox{AP}2)_{\varepsilon}(\tilde{u}_{0\varepsilon}, s, \eta, b)$ on $[0, T]$ for each $\varepsilon>0$.
By letting $\varepsilon \to 0$ we show the existence of a solution $\tilde{u}$ of $(\mbox{AP}2)(\tilde{u}_0, s, \eta, b)$ on $[0, T]$.  First, by Lemma \ref{lem6}
%Similarly to the derivation of (\ref{3-15-8}), we have 
%\begin{align}
%\label{3-26}
%& \frac{1}{2} |\tilde{u}_{\varepsilon}(t)|^2_{H} + \frac{1}{4l^2}\int_0^t |\tilde{u}_{{\varepsilon}y}(\tau)|^2_{H} dy \nonumber \\
%%\leq & \frac{1}{2} |\tilde{u}_0|^2_{H} + \frac{a_0 \alpha}{2} |\tilde{u}_0(1)|^2\varepsilon \nonumber \\
%& + \int_0^t \biggl (F(\tau) + l^2 \left(\frac{a_0 \alpha}{2} C_e\right)^2 + \frac{a_0\alpha}{2}C_e\biggr)|\tilde{u}_{\varepsilon}(\tau)|^2_{H} d\tau + %\frac{\beta b^*}{2s_0}T \mbox{ for }t\in[0, T],
%\end{align}
%where
%\begin{align*}
%F(t)=\left( \frac{|s_{t}(t)|^2}{{2}} + (\beta b^* C_e)^2+ \frac{\beta b^* C_e}{2a} \right)+ \biggl (s^2(t) \left(\frac{a_0\alpha}{2} C_e\right)^2 + %\frac{a_0\alpha}{2} \biggr). 
%\end{align*}
%Since $F\in L^1(0, T)$ and (\ref{3-26}), 
we see that $\{ \tilde{u}_{\varepsilon} \}$ is bounded in $L^{\infty}(0, T; H)\cap L^2(0, T; X)$. Next, for $z\in X$, it holds that 
\begin{align}
\label{3-27}
& \biggl| \int_0^1 \tilde{u}_{\varepsilon t}(t) z dy \biggr| \nonumber \\
= & \biggl| -\frac{1}{s^2(t)} \biggl(\int_0^1 \tilde{u}_{\varepsilon y}(t) z_y dy\biggr) -\frac{a_0}{s(t)} ((\sigma(\tilde{u}_{\varepsilon}(t, 1))^2-\alpha\sigma((\rho_{\varepsilon}*\eta)(t, 1))s(t))z(1)  \nonumber \\
& + \frac{1}{s(t)}\beta(b(t)-\gamma \tilde{u}_{\varepsilon}(t, 0))z(0)  + \int_0^1\frac{y s_t(t)}{s(t)} \tilde{u}_{\varepsilon y}(t) z dy \biggr| \nonumber \\
\leq & \frac{1}{a^2} |\tilde{u}_{\varepsilon y}(t)|_{H} |z_y|_{H} + \frac{a_0}{a} |\tilde{u}_{\varepsilon}(t, 1)|^2 |z(1)| + a_0 \alpha |(\rho_{\varepsilon}*\eta)(t, 1)||z(1)| \nonumber \\
& + \frac{\beta}{a} b^*|z(0)| + \frac{\beta}{a} \gamma |\tilde{u}_{\varepsilon}(t, 0)||z(0)| 
+ \frac{|s_t(t)|}{a} \biggl (|\tilde{u}_{\varepsilon}(t, 1)||z(1)| + |\tilde{u}_{\varepsilon}(t)|_H(|z_y|_H + |z|_H) \biggr) \nonumber \\
& \mbox{ for a.e. }t\in [0, T]. 
\tag{3.31}
\end{align}
Here, we note that 
\begin{align*}
%\label{3-30-1}
\int_0^1\frac{y s_t(t)}{s(t)} \tilde{u}_{\varepsilon y}(t) zdy = \frac{s_t(t)}{s(t)} \biggl (\tilde{u}_{\varepsilon}(t, 1)z(1)-\int_0^1 \tilde{u}_{\varepsilon}(t)(y z_y +z)dy \biggr). 
\end{align*}
By the estimate (\ref{3-27}) and (\ref{1d}) we infer that $\{ \tilde{u}_{\varepsilon t}\}$ is bounded in $L^2(0, T; X^*)$. 
Therefore, we take a subsequence $\{ {\varepsilon}_i \}\subset \{\varepsilon \}$ such that for some $\tilde{u}\in W^{1, 2}(0, T; X^*)\cap L^{\infty}(0, T; H)\cap L^2(0, T; X)$, $\tilde{u}_{{\varepsilon}_i} \to \tilde{u}$ weakly in $W^{1,2}(0, T; X^*) \cap L^2(0, T; X)$, weakly-* in $L^{\infty}(0, T;H)$ as $i \to \infty$.  Also, by Aubin's compactness theorem, we see that $\tilde{u}_{{\varepsilon}_i} \to \tilde{u}$ in $L^2(0, T; H)$ as $i \to \infty$. 

Now, we prove that the limit function $\tilde{u}$ is a solution of $(\mbox{AP}2)(\tilde{u}_0, s, \eta, b)$ on $[0, T]$ satisfying $\tilde{u}\in W^{1, 2}(0, T; X^*)\cap V(T)$, (S''2) and (S3). Let $z\in V(T)$. Then, it holds that 
\begin{align}
\label{3-31}
& \int_0^T \int_0^1 \tilde{u}_{\varepsilon t}(t) z(t) dy dt + \int_0^T \frac{1}{s^2(t)} \biggl(\int_0^1 \tilde{u}_{\varepsilon y}(t) z_y(t) dy\biggr)dt \nonumber \\
& +  \int_0^T \frac{a_0}{s(t)} ((\sigma(\tilde{u}_{\varepsilon}(t, 1))^2-\alpha \sigma((\rho_{\varepsilon}*\eta)(t, 1)) s(t))z(t, 1) dt \nonumber \\
& -\int_0^T\frac{1}{s(t)}\beta(b(t)-\gamma \tilde{u}_{\varepsilon}(t, 0))z(t, 0) dt = \int_0^T \int_0^1\frac{y s_t(t)}{s(t)} \tilde{u}_{\varepsilon y}(t) z(t) dy dt.
\tag{3.32}
\end{align}
From the weak convergence, it is easy to see that 
\begin{align*}
& \int_0^T\int_0^1 \tilde{u}_{\varepsilon_i t}(t) z(t) dy dt \to \int_0^T \langle \tilde{u}_t(t), z(t)\rangle_X dt, \\
& \int_0^T \frac{1}{s^2(t)} \biggl(\int_0^1 \tilde{u}_{\varepsilon_i y}(t) z_y(t) dy\biggr)dt \to \int_0^T \frac{1}{s^2(t)} \biggl(\int_0^1 \tilde{u}_{y}(t) z_y(t) dy\biggr)dt \mbox{ as } i \to \infty.
\end{align*}
The third term of the right-hand side of (\ref{3-31}) is as follows:
\begin{align*}
&\biggl| \int_0^T \frac{a_0}{s(t)} ((\sigma(\tilde{u}_{\varepsilon}(t, 1))^2-\alpha \sigma((\rho_{\varepsilon}*\eta)(t, 1)) s(t))z(t, 1) dt \\
&  \hspace{3.5cm}- \int_0^T \frac{a_0}{s(t)} ((\sigma(\tilde{u}(t, 1))^2-\alpha \sigma(\eta(t, 1))s(t))z(t, 1) dt \biggr| \\
\leq & \frac{a_0}{a}\biggl( \int_0^T|\tilde{u}_{\varepsilon}(t, 1)-\tilde{u}(t, 1)|^2dt \biggr)^{1/2} \biggl( \int_0^T2(|\tilde{u}_{\varepsilon}(t, 1)|^2+|\tilde{u}(t)|^2)|z(t, 1)|^2 dt \biggr)^{1/2}. \\
%& + \frac{1}{a}\biggl( \int_0^T|\tilde{u}_{\varepsilon}(t, 1)-\tilde{u}(t, 1)|^2dt \biggr)^{1/2}\biggl( \int_0^T |\tilde{u}(t, 1)|^2|z(t, 1)|^2 dt \biggr)^{1/2} \\
& + a_0 \alpha \int_0^T |(\rho_{\varepsilon}*\eta)(t, 1)-\eta(t, 1)||z(t, 1)|dt
\end{align*}
Here, by (\ref{1d}) we note that it holds that 
\begin{align}
\label{3-32}
&\int_0^T|\tilde{u}_{\varepsilon}(t, z)-\tilde{u}(t, z)|^2dt \nonumber \\
\leq & C_e \biggl( \int_0^T|\tilde{u}_{\varepsilon}(t)-\tilde{u}(t)|^2_{X}\biggr)^{1/2}
\biggl( \int_0^T|\tilde{u}_{\varepsilon}(t)-\tilde{u}(t)|^2_{H}\biggr)^{1/2} \mbox{ for }z=0, 1,
\tag{3.33}
\end{align}
and 
\begin{align*}
 &\int_0^T|\tilde{u}_{\varepsilon}(t, 1)|^2|z(t, 1)|^2 dt \leq C^2_e \int_0^T |\tilde{u}_{\varepsilon}(t)|_{X}|\tilde{u}_{\varepsilon}(t)|_{H}
 |z(t)|_{X}|z(t)|_{H} dt \\
& \leq C^2_e |\tilde{u}_{\varepsilon}|_{L^{\infty}(0, T; H)} |z|_{L^{\infty}(0, T; H)} \biggl( \int_0^T|\tilde{u}_{\varepsilon}(t)|^2_{X}\biggr)^{1/2}\biggl( \int_0^T|z(t)|^2_{X}\biggr)^{1/2}.
\end{align*}
Since $(\rho_{\varepsilon}*\eta)(t, 1) \to \eta(t, 1)$ in $L^2(0, T)$ as $\varepsilon \to 0$,  
by the boundedness of $\tilde{u}_{\varepsilon}$ in $L^2(0, T; X)$, $\tilde{u}\in L^2(0, T; X)$ and the strong convergence in $L^2(0, T; H)$ we see that 
\begin{align*}
& \int_0^T \frac{a_0}{s(t)} ((\sigma(\tilde{u}_{\varepsilon_i}(t, 1))^2-\alpha \sigma((\rho_{\varepsilon_i}*\eta)(t, 1))s(t))z(t, 1) dt \\
\to & \int_0^T \frac{a_0}{s(t)} ((\sigma(\tilde{u}(t, 1))^2-\alpha \sigma(\eta(t, 1))s(t))z(t, 1) dt \mbox{ as } i \to \infty.
\end{align*}
What concerns the forth term on the right-hand side of (\ref{3-31}), by (\ref{3-32}) it follows that 
\begin{align*}
\int_0^T\frac{1}{s(t)}\beta(b(t)-\gamma \tilde{u}_{\varepsilon_i}(t, 0))z(t, 0) dt
\to \int_0^T\frac{1}{s(t)}\beta(b(t)-\gamma \tilde{u}(t, 0))z(t, 0) dt \mbox{ as }i \to \infty.
\end{align*}
Moreover, since $y\frac{s_t}{s}z\in L^2(0, T; H)$, by  weak convergence it holds that 
\begin{align*}
& \int_0^T \int_0^1\frac{y s_t(t)}{s(t)} \tilde{u}_{\varepsilon_i y}(t) z(t) dy dt = \int_0^T (\tilde{u}_{\varepsilon_i y}(t), y\frac{s_t(t)}{s(t)}z(t))_{H} dt \\
& \to \int_0^T (\tilde{u}_{y}(t), y\frac{s_t(t)}{s(t)}z(t))_{H} dt = \int_0^T \int_0^1\frac{y s_t(t)}{s(t)} \tilde{u}_{y}(t) z(t) dy dt
\mbox{ as }i \to \infty.
\end{align*}
Therefore, by the limiting process $i \to \infty$ in (\ref{3-31}) we see that the limit function $\tilde{u}$ is a solution of $(\mbox{AP}2)(\tilde{u}_0, s, \eta, b)$ on $[0, T]$. Also, the solution $\tilde{u}$ is unique. Indeed, let $\tilde{u}_1$ and $\tilde{u}_2$ be a solution of $(\mbox{AP}2)(\tilde{u}_0, s, \eta, b)$ on $[0, T]$ and put $\tilde{u}=\tilde{u}_1-\tilde{u}_2$. Then, (S''2) implies that 
\begin{align}
\label{3-33}
& \langle \tilde{u}_t, z \rangle_X + \frac{1}{s^2} \biggl(\int_0^1 \tilde{u}_{y}z_y dy\biggr)\nonumber \\
& +  \frac{a_0}{s} \biggl((\sigma(\tilde{u}_1(\cdot, 1))^2-\alpha\sigma( \eta (\cdot, 1))s)-((\sigma(\tilde{u}_2(\cdot, 1))^2-\alpha \sigma(\eta(\cdot, 1))s) \biggr)z(1)\nonumber \\
& +\frac{1}{s}\beta \gamma (\tilde{u}_1(\cdot, 0)-\tilde{u}_2(\cdot, 0))z(0) = \int_0^1\frac{y s_t}{s} \tilde{u}_yz dy
\mbox{ for }z\in X \mbox{ a.e. on }[0, T].
\tag{3.34}
\end{align}
We take $z=\tilde{u}$ in (\ref{3-33}). Then, we have
\begin{align}
\label{3-34}
& \frac{1}{2} \frac{d}{dt}|\tilde{u}(t)|^2_{H} + \frac{1}{s^2} \int_0^1|\tilde{u}_y(t)|^2 dy + \frac{a_0}{s(t)} \biggl ((\sigma(\tilde{u}_1(t, 1))^2-(\sigma(\tilde{u}_2(t, 1))^2 \biggr)\tilde{u}(t, 1) \nonumber \\
& %- a_0\alpha |\tilde{u}(t, 1)|^2 
+ \frac{\beta \gamma}{s(t)} |\tilde{u}(t, 0)|^2=  \int_0^1\frac{y s_t(t)}{s(t)} \tilde{u}_y(t)\tilde{u}(t) dy \mbox{ for a.e. }t\in [0, T].
\tag{3.35}
\end{align}
The third and forth terms  in  the left-hand side of (\ref{3-34}) are non-negative. 
%\begin{align*}
%& \frac{a_0}{s(t)} \biggl (\tilde{u}_1(t, 1)\sigma(\tilde{u}_1(t, 1))-\tilde{u}_2(t, 1)\sigma(\tilde{u}_2(t, 1)) \biggr)\tilde{u}(t, 1) \\
%\geq & -\frac{a_0}{s(t)}|\tilde{u}_2(t, 1)||\tilde{u}(t, 1)|^2 \geq -\frac{a_0}{s(t)}|\tilde{u}_2(t, 1)|C_e (|\tilde{u}_y(t)|_{H}|\tilde{u}(t)|_{H} + |\tilde{u}(t)|^2_{H}) \\
%\geq & -\frac{1}{4s^2(t)}|\tilde{u}_y(t)|^2_H - \biggl( (a_0C_e|\tilde{u}_2(t, 1)|)^2 + \frac{a_0C_e}{a} |\tilde{u}_2(t, 1)|\biggr)|\tilde{u}(t)|^2_H. 
%\end{align*}
%Here, by (\ref{1d}) and (\ref{3-15-11}) we infer that $\tilde{u}_2(\cdot, 1)\in L^2(0, T)$. 
For % the fourth term in the left-hand side of (\ref{3-34}) and 
the right-hand side of (\ref{3-34}), we have 
%\begin{align*}
%& a_0 \alpha|\tilde{u}(t, 1)|^2 \leq a_0 \alpha C_e (|\tilde{u}_y(t)|_{H}|\tilde{u}(t)|_{H} + |\tilde{u}(t)|^2_{H}) \\
%\leq & \frac{1}{4s^2(t)}|\tilde{u}_y(t)|^2_{H} + \biggl( s^2(t)(C_e a_0 \alpha)^2 + C_e a_0 \alpha \biggr) |\tilde{u}(t)|^2_{H},
%\end{align*}
%and 
\begin{align*}
\int_0^1\frac{y s_t(t)}{s(t)} \tilde{u}_y(t)\tilde{u}(t) dy\ \leq \frac{|s_t(t)|}{s(t)}|\tilde{u}_y(t)|_{H} |\tilde{u}(t)|_{H} \leq \frac{1}{4s^2(t)}|\tilde{u}_y(t)|^2_{H} + |s_t(t)|^2|\tilde{u}(t)|^2_{H}.
\end{align*}
Therefore, by adding these estmates to (\ref{3-34}) and Gronwall's inequality, we have the uniqueness of a solution $\tilde{u}$ to $(\mbox{AP}2)(\tilde{u}_0, s, \eta, b)$ on $[0, T]$.
\end{proof}

%%%%%%%%%%%%%%%%%%%%%%%%%%%%%%%%%%%%%%%%%%%%%%%%%%%%%%%%%%%%%%%%%%%%%%%%%%%%%%%%%%%%%%%%%%%%%
%
%
% Lemma 3.6 // 
%
%
%%%%%%%%%%%%%%%%%%%%%%%%%%%%%%%%%%%%%%%%%%%%%%%%%%%%%%%%%%%%%%%%%%%%%%%%%%%%%%%%%%%%%%%%%%%%%

\begin{lemma}
\label{lem5}
Let $T>0$ and $s\in W^{1,\infty}(0, T)$ with $s>0$ on $[0, T]$. Then, $(\mbox{AP}1)(\tilde{u}_{0}, s, b)$ has a unique solution $\tilde{u}$ on $[0, T]$. 
\end{lemma}
\begin{proof}
%Note that for given $\eta \in V(T)$, $\rho_{\varepsilon}*\eta \in W^{1,2}(0, T)$, where $\rho_{\varepsilon}$ is the same function in (\ref{3-5}). Hence, %by considering  $(\mbox{AP})(\tilde{u}_0, s, v, b)$ as $v=\rho_{\varepsilon}*\eta$, 
From Lemma \ref{lem7}, we see that $(\mbox{AP}2)(\tilde{u}_{0}, s, \eta, b)$ has a solution $\tilde{u}$ on $[0, T]$ such that $\tilde{u}\in W^{1, 2}(0, T; X^*)\cap L^{\infty}(0, T; H)\cap L^2(0, T; X)$.
Define a solution operator $\delta _T(\eta)=\tilde{u}$, where $\tilde{u}$ is a unique solution of  $(\mbox{AP}2)(\tilde{u}_{0}, s, \eta, b)$ for given $\eta \in V(T)$. Let put $\delta_T(\eta_i)=\tilde{u}_i$ for $i=1, 2$ and $\eta=\eta_1-\eta_2$ and $\tilde{u}=\tilde{u}_1-\tilde{u}_2$. 
Then, by (S''2) it holds that 
\begin{align}
\label{ap-5}
& \frac{1}{2} \frac{d}{dt}|\tilde{u}(t)|^2_{H} + \frac{1}{s^2} \int_0^1|\tilde{u}_y(t)|^2 dy \nonumber \\
& +\frac{a_0}{s(t)} \biggl((\sigma(\tilde{u}_1(t, 1))^2-\alpha \sigma(\eta_1 (t, 1))s(t))-((\sigma(\tilde{u}_2(t, 1))^2-\alpha \sigma(\eta_2(t, 1))s(t)) \biggr)\tilde{u}(t, 1) \nonumber \\
& %- a_0\alpha |\tilde{u}(t, 1)|^2 
+ \frac{\beta \gamma}{s(t)} |\tilde{u}(t, 0)|^2=  \int_0^1\frac{y s_t(t)}{s(t)} \tilde{u}_y(t)\tilde{u}(t) dy \mbox{ for a.e. }t\in [0, T].
\tag{3.36}
\end{align}
For the third term $I_3$ in the left-hand side of (\ref{ap-5}), by the monotonicity of $\sigma$ we have that $I_3 \geq  -a_0\alpha |\eta(t, 1)| |\tilde{u}(t, 1)|$.
%\begin{align}
%\label{ap-6}
%\end{align}
From (\ref{1d}) it holds that %we estimate the second term of the right-hand side of (\ref{ap-6}) as follows:
\begin{align*}
%\label{ap-7}
&|\eta(t, 1)|| \tilde{u}(t, 1)| \leq C^{1/2}_e |\eta(t, 1)|(|\tilde{u}_y(t)|^{1/2}_{H}|\tilde{u}(t)|^{1/2}_{H} + |\tilde{u}(t)|_{H} )
\end{align*}
and  
\begin{align}
\label{ap-8}
& \int_0^t |\eta(\tau, 1)|| \tilde{u}(\tau, 1)| d\tau \nonumber \\
\leq & C^{1/2}_e \bigg( |\tilde{u}|^{1/2}_{L^{\infty}(0, T_1; H)} \int_0^t |\eta(t, 1)| |\tilde{u}_y(\tau)|^{1/2}_{H} d\tau + |\tilde{u}|_{L^{\infty}(0, T_1; H)} \int_0^t |\eta(t, 1)| d\tau \biggr) \nonumber \\
\leq & C^{1/2}_e\biggl(|\tilde{u}|^{1/2}_{L^{\infty}(0, T_1; H)} T^{1/4}_1 \biggl (\int_0^t |\tilde{u}_y(\tau)|^2_{H} \biggr)^{1/4} + |\tilde{u}|_{L^{\infty}(0, T_1; H)} T^{1/2}_1 \biggr) \biggl (\int_0^t |\eta(t, 1)|^2\biggr)^{1/2} . 
\tag{3.37}
\end{align}
%Also, by  $\tilde{u}_i\in L^{\infty}(0, T; X)$ for $i=1$, 2 and (\ref{1d}) we note that $\tilde{u}_i(\cdot, 1)\in L^{\infty}(0, T)$ for $i=1$, 2. 
Let $T_1\in (0, T]$ and we integrate (\ref{ap-5}) over $[0, t]$ for any $t\in [0, T_1]$. Then, by (\ref{3-14-1}) and (\ref{ap-8}) we obtain 
\begin{align}
\label{ap-9}
& \delta \biggl (|\tilde{u}(t)|^2_{H} + \int_0^t |\tilde{u}_y(t)|^2_{H} \biggr) \nonumber \\
%\frac{a_0}{a}C_e|\tilde{u}_2(\cdot, 1)|_{L^{\infty}(0, T)}T^{1/2}_1(1+T^{1/2}) |\tilde{u}|^2_{V(T_1)} 
\leq & a_0\alpha C^{1/2}_e T^{1/4}_1(1+T^{1/4}_1) \biggl (\int_0^t |\eta(t, 1)|^2 \biggr)^{1/2} |\tilde{u}|_{V(T_1)} + \frac{|s_t|_{L^{\infty}(0, T_1)}}{a} T^{1/2}_1 |\tilde{u}|^2_{V(T_1)}, 
\tag{3.38}
\end{align}
where $\delta=\min \{1/2, 1/l^2 \}$, where $l=\max_{0\leq t\leq T}|s(t)|$. %Put the coefficient of $|\tilde{u}|^2_{V(T_1)}$ by $C(T)$. 
%Also, we note that it holds that 
%\begin{align}
%\label{ap-10}
%&\int_0^ t |\eta(\tau, 1)|^2 d\tau \leq C_e \int_0^t (|\eta_y(\tau)|_{H}|\eta(\tau)|_{H} + |\eta(\tau)|^2_H) dt \nonumber \\
%& \leq C_e \biggl(|\eta|_{L^{\infty}(0, T_1; H)} T^{1/2} \biggl( \int_0^t |\eta_y(\tau)|^2_H d\tau \biggr)^{1/2} + T|\eta|^2_{L^{\infty}(0, T_1; H)}  \biggr)  \nonumber \\
%& \leq C_eT^{1/2}(1+T^{1/2}) |\eta|^2_{V(T_1)}
%\end{align}
Finally, by (\ref{ap-9}) we have 
\begin{align}
\left (\delta -\frac{|s_t|_{L^{\infty}(0, T_1)}}{a}T^{1/2}_1 \right)|\tilde{u}|_{V(T_1)} \leq a_0 \alpha C^{1/2}_eT^{1/4}_1(1+T^{1/4}_1)|\eta|_{V(T_1)}.
\tag{3.39}
%\biggl (\int_0^t |\eta(\tau)|^2_{H^1(0, 1)} \biggr)^{1/2} %+ T^{1/2}_1 \biggl( \int_0^t |\eta(\tau)|^2_{H^1(0, 1)}|d\tau \biggr)^{1/2} \nonumber \\
%& + \frac{|s_t(t)|_{L^{\infty}(0, T_1)}}{s_0} T^{1/2}_1 \biggl( \int_0^t |\tilde{u}_y(\tau)|^2_{L^2(0, 1)} \biggr)^{1/2}
\end{align}
Now, we take $T^*_1$ such that $\delta -\frac{|s_t|_{L^{\infty}(0, T_1)}}{a}(T^*_1)^{1/2}=\frac{\delta}{2}$ and see that $T_1\leq T^*_1$ such that $\frac{2}{\delta} a_0 \alpha C^{1/2}_eT^{1/4}_1(1+T^{1/4}_1)<1$. Hence, Banach's fixed point theorem guarantees that 
%From this result, we infer that for some $T_1$, $\delta_{T_1}$ is a contraction mapping in $V(T_1)$. 
there exists $\tilde{u}\in V(T_1)$ such that $\delta_{T_1}(\tilde{u})=\tilde{u}$. Thus, we see that $(\mbox{AP}1)(\tilde{u}_{0}, s, b)$ has a solution $\tilde{u}$ on $[0, T_1]$. Here, $T_1$ is independent of the choice of the initial data. Therefore, %similarly to the proof of Lemma \ref{lem2}, 
by repeating the local existence argument, we have a unique solution $\tilde{u}$ of $(\mbox{AP}1)(\tilde{u}_{0}, s, b)$ on the whole interval $[0, T]$. Thus, we see that Lemma \ref{lem5} holds.
\end{proof}

Here, for given $s\in W^{1,\infty}(0, T)$ with $s>0$ on $[0, T]$ we show that a solution $\tilde{u}$ of $(\mbox{AP}1)(\tilde{u}_0, s, b)$ is non-negative and bounded on $Q(T)$. 

%%%%%%%%%%%%%%%%%%%%%%%%%%%%%%%%%%%%%%%%%%%%%%%%%%%%%%%%%%%%%%%%%%%%%%%%%%%%%%%%%%%%%%%%%%%%
%
%
% lemma 3.7// 
%
%
%%%%%%%%%%%%%%%%%%%%%%%%%%%%%%%%%%%%%%%%%%%%%%%%%%%%%%%%%%%%%%%%%%%%%%%%%%%%%%%%%%%%%%%%%%%%%
\begin{lemma}
\label{lem9}
Let $T>0$, $s\in W^{1,\infty}(0, T)$ with $s>0$ on $[0, T]$ and $\tilde{u}$ be a solution of $(\mbox{AP}1)(\tilde{u}_0, s, b)$ on $[0, T]$. Then, it holds that 
\begin{align*}
0\leq \tilde{u}(t) \leq u^*(T):=\max \{ \alpha l(T), \frac{b^*}{\gamma} \} \mbox{ on } [0, 1] \mbox{ for } t\in [0, T],
\end{align*}
where $l(T)=\max_{0\leq t \leq T}|s(t)|$.
\end{lemma}

\begin{proof}
By (S'2), we note that it holds that 
\begin{align}
\label{3-42}
& \langle \tilde{u}_t, z \rangle_X + \int_0^1 \frac{1}{s^2} \tilde{u}_{y} z_y dy + \frac{a_0}{s} \sigma(\tilde{u}(\cdot, 1))(\sigma(\tilde{u}(\cdot, 1))-\alpha s) z(1)\nonumber \\
&-\frac{1}{s}\beta (b(\cdot)-\gamma \tilde{u}(\cdot, 0))z(0) = \int_0^1\frac{y s_{t}}{s} \tilde{u}_{y}z dy \mbox{ for }z\in X\mbox{ a.e. on }[0, T].
\tag{3.40}
\end{align}
First, we prove that $\tilde{u}(t)\geq 0$ on $[0, 1]$ for $t\in [0, T]$. By taking $z=-[-\tilde{u}]^+$ in (\ref{3-42}) we have 
\begin{align}
\label{3-43}
& \frac{1}{2} \frac{d}{dt}|[-\tilde{u}(t)]^+|^2_{H}+ \frac{1}{s^2(t)} \int_0^1 |[-\tilde{u}(t)]^+_y|^2 dy \nonumber \\
& -\frac{a_0}{s(t)} \sigma(\tilde{u}(\cdot, 1))(\sigma(\tilde{u}(\cdot, 1))-\alpha s(t))[-\tilde{u}(t, 1)]^+ +\frac{1}{s(t)}\beta (b(t)-\gamma \tilde{u}(\cdot, 0))[-\tilde{u}(t, 0)]^+ \nonumber \\
& = -\int_0^1\frac{y s_t(t)}{s(t)} \tilde{u}_y(t)[-\tilde{u}(t)]^+dy
\mbox{ a.e. on }[0, T].
\tag{3.41}
\end{align}
Here, the third term in the left-hand side of (\ref{3-43}) is equal to 0 and the forth term in the left-hand side of (\ref{3-43}) is non-negative. 
%Also, from (1.4) and (\ref{1d}) we have that 
%\begin{align*}
%& -\frac{a_0}{s(t)}\sigma(\tilde{u}(\cdot, 1))(\sigma(\tilde{u}(\cdot, 1)-\alpha s(t))[-\tilde{u}(t, 1)]^+ \\
%& \geq -a_0 \alpha |[-\tilde{u}(t, 1)]^+|^2 \\
%& \geq -a_0 \alpha C_e (|[-\tilde{u}(t, 1)]^+_y|_{H} |[-\tilde{u}(t, 1)]^+|_{H} +|[-\tilde{u}(t, 1)]^+|^2_{H}) \\
%& \geq -\frac{1}{4s^2(t)}|[-\tilde{u}(t, 1)]^+_y|^2_{H} -( l^2(a_0 \alpha C_e)^2 + a_0\alpha C_e)|[-\tilde{u}(t, 1)]^+|^2_{H},
%\end{align*}
Also, we obtain that 
\begin{align*}
& \int_0^1\frac{y s_t(t)}{s(t)} [-\tilde{u}(t)]^+_y[-\tilde{u}(t)]^+dy  
%& \leq a_0 \alpha |[-\tilde{u}(t)]^+_y|_{H}|[-\tilde{u}(t)]^+|_{H} \\
\leq \frac{1}{2s^2(t)}|[-\tilde{u}(t)]^+_y|^2_{H} + \frac{(s_t(t))^2}{2}|[-\tilde{u}(t)]^+|^2_{H}, 
\end{align*}
Then, we have 
\begin{align*}
\frac{1}{2} \frac{d}{dt}|[-\tilde{u}(t)]^+|^2_{H}+ \frac{1}{2s^2(t)} \int_0^1 |[-\tilde{u}(t)]^+_y|^2 dy \leq \frac{(s_t(t))^2}{2}|[-\tilde{u}(t)]^+|^2_{H} \mbox{ for a.e. }t\in [0, T].
\end{align*}
Therefore, by Gronwall's inequality and the assumption that $\tilde{u}_0 \geq 0$ on $[0, 1]$, we conclude that $\tilde{u}(t) \geq 0$ on $[0, 1]$ for $t\in [0, T]$. 

Next, we show that a solution $\tilde{u}$ of $(\mbox{AP}1)(\tilde{u}_0, s, b)$ has a upper bound $u^*(T)$. Put $U(t, y)=[\tilde{u}(t, y)-u^*(T)]^+$ for $y \in [0, 1]$ and $t\in [0, T]$. Then, it holds that 
\begin{align}
\label{3-44}
& \frac{1}{2} \frac{d}{dt}|U(t)|^2_{H}+ \frac{1}{s^2(t)} \int_0^1 |U_y(t)|^2 dy + \frac{a_0}{s(t)} \sigma(\tilde{u}(\cdot, 1))(\sigma(\tilde{u}(\cdot, 1))-\alpha s(t))U(t, 1) \nonumber \\
& -\frac{1}{s(t)}\beta (b(t)-\gamma \tilde{u}(\cdot, 0))U(t, 0)= \int_0^1\frac{y s_t(t)}{s(t)} \tilde{u}_y(t)U(t) dy
\mbox{ for a.e. }t\in [0, T].
\tag{3.42}
\end{align}
Here, by $\tilde{u}(t)\geq 0$ on $[0, 1]$ for $t\in [0,T]$ we note that  $a_0\sigma(\tilde{u}(\cdot, 1))(\sigma(\tilde{u}(t, 1))-\alpha s(t))=a_0\tilde{u}(t, 1)(\tilde{u}(t, 1)-\alpha s(t))$. Then, by $u^*(T)\geq \alpha l(T) \geq \alpha s(t)$ for $t\in [0, T]$, it holds that 
\begin{align*}
\frac{a_0}{s(t)} \tilde{u}(t, 1)(\tilde{u}(t, 1)-\alpha s(t)) U(t, 1)\geq \frac{a_0}{s(t)}\tilde{u}(t, 1)(u^*(T)-\alpha s(t))U(t, 1)\geq 0.
\end{align*}
%For the third term in the left-hand side of (\ref{4-20}) we can write 
%By noting from $\sup_{r\in \mathbb{R}}\varphi(r)\leq |h|_{L^{\infty}(0, T)}H^{-1}$ in (A4) that
%\begin{align*}
%& -u_z(t, s(t))U(t, s(t)) = 
%\frac{1}{s(t)}\tilde{u}(t, 1)s_t(t) U(t, 1) & = \frac{1}{s(t)}(s_t(t) \left(\tilde{u}(t, 1)-u^* \right)U(t, 1) + s_t(t) u^* U(t, 1)) \\  
%& =\frac{s_t(t)}{s(t)}|U(t, s(t))|^2 + \frac{s_t(t)}{s(t)}u^* U(t, s(t)).
%\end{align*}
Also, by (\ref{1-3}) and $b\leq b^*$, we observe that 
\begin{align}
\label{3-45}
& -\frac{1}{s(t)}\beta(b(t)-\gamma \tilde{u}(t, 0))U(t, 0) = \frac{1}{s(t)}\beta (\gamma \tilde{u}(t, 0)-b^* + b^*-b(t))U(t, 0) \nonumber \\
\geq  & \frac{\beta \gamma}{s(t)} |U(t, 0)|^2 + \frac{\beta}{s(t)}(b^*-b(t))U(t, 0) \geq 0.
\tag{3.43}
\end{align}
By applying the above two results to (\ref{3-44}) we obtain that 
\begin{align*}
& \frac{1}{2}\frac{d}{dt}\int_0^1|U(t)|^2 dy + \frac{1}{2s^2(t)}\int_0^1 |U_y(t)|^2 dy \leq \frac{(s_t(t))^2}{2}|U(t)|^2_H \mbox{ for a.e. }t\in[0, T].
%\label{4-35}
\end{align*}
%Here, by $s_t(t)=a_0(\sigma(u(t, s(t)))-\alpha s(t))$ we notice that 
%\begin{align*}
%\frac{s_t(t)}{2}|U(t, s(t))|^2 \geq \frac{a_0(u^*-\alpha s(t))}{2}|U(t, s(t))|^2 \geq \frac{a_0(u^*-\alpha M)}{2}|U(t, s(t))|^2\geq 0.
%\end{align*}
%This implies that 
This result and the assumption that $\tilde{u}_0 \leq b^*/\gamma $ on $[0, 1]$ implies that $\tilde{u}(t) \leq u^*(T)$ on $[0, 1]$ for $t\in [0, T]$. Thus, Lemma \ref{lem9} is proven.
\end{proof}

At the end of this section, we relax the condition $s\in W^{1, \infty}(0, T)$, namely, for given $s\in W^{1,2}(0, T)$ with $s>0$ on $[0, T]$, we construct a solution to $(\mbox{AP}1)(\tilde{u}_0, s, b)$.

%%%%%%%%%%%%%%%%%%%%%%%%%%%%%%%%%%%%%%%%%%%%%%%%%%%%%%%%%%%%%%%%%%%%%%%%%%%%%%%%%%%%%%%%%%%%
%
%
% lemma 3.8// 
%
%
%%%%%%%%%%%%%%%%%%%%%%%%%%%%%%%%%%%%%%%%%%%%%%%%%%%%%%%%%%%%%%%%%%%%%%%%%%%%%%%%%%%%%%%%%%%%%
\begin{lemma}
\label{lem8}
Let $T>0$ and $s\in W^{1,2}(0, T)$ with $s>0$ on $[0, T]$. If (A1)-(A3), then, $(\mbox{AP}1)(\tilde{u}_0, s, b)$ has a unique solution $\tilde{u}$ on $[0, T]$. 
\end{lemma}

\begin{proof}
For given $s\in W^{1, 2}(0, T)$ with $s>0$ on $[0, T]$, we choose a sequence $\{s_n\} \subset W^{1, \infty}(0, T)$ and $l$, $a>0$ satisfying $a \leq s_n\leq l$ on $[0, T]$ for each $n\in \mathbb{N}$, 
$s_n \to s$ in $W^{1,2}(0, T)$ as $n \to \infty$. By Lemma \ref{lem5} we can take a sequence $\{\tilde{u}_n\}$ of solutions to $(\mbox{AP}1)(\tilde{u}_0, s_n, b)$ on $[0, T]$. Let $z\in X$. Then, it holds that 
\begin{align}
\label{3-40}
& \langle \tilde{u}_{nt}, z \rangle_X + \frac{1}{s^2_{n}} \biggl(\int_0^1 \tilde{u}_{ny} z_y dy\biggr) \nonumber \\
& +  \frac{a_0}{s_{n}} \sigma(\tilde{u}_{n}(\cdot, 1)) (\sigma(\tilde{u}_{n}(\cdot, 1))-\alpha s_{n})z(1) \nonumber \\
& -\frac{1}{s_{n}}\beta(b-\gamma \tilde{u}_{n}(\cdot, 0))z(0) = \int_0^1\frac{y s_{nt}}{s_{n}} \tilde{u}_{ny}zdy 
\mbox{ a.e. on }[0, T].
\tag{3.44}
\end{align}
We take $z=\tilde{u}_n$ in (\ref{3-40}). Then, simiarly to the proof of Lemma \ref{lem6} we derive  
\begin{align*}
%\label{3-15-4}
& \frac{1}{2}\frac{d}{dt}|\tilde{u}_{n}(t)|^2_{H} + \frac{1}{4s^2_{n}(t)}\int_0^1|\tilde{u}_{ny}(\tau)|^2 dy \\
\leq & \left( |s_{nt}(t)|^2 + \frac{(\beta b^* C_e)^2}{4}+ \frac{\beta b^* C_e}{2a} \right)|\tilde{u}_{n}(t)|^2_{H} + \frac{\beta b^*}{2a} \\
& + \biggl (s^2_{n}(t) \left(a_0\alpha C_e\right)^2 + a_0\alpha C_e \biggr)|\tilde{u}_{n}(t)|^2_{H}
 \mbox{ for a.e. }t\in[0, T].
\end{align*}
From this, we infer that $\{ \tilde{u}_n \}$ is bounded in $L^{\infty}(0, T; H)\cap L^2(0, T; X)$. Also, by using Lemma \ref{lem9}, it holds that $0\leq \tilde{u}_n \leq u^*$ on $Q(T)$ for each $n\in \mathbb{N}$, where $u^*=\max \{ \alpha l, \frac{b^*}{\gamma} \}$. Using this result, the boundedness of $\{ \tilde{u}_n \}$ in $V(T)$ and (\ref{ap-5}), %(\ref{3-27}), 
we see that $\{ \tilde{u}_{nt}\}$ is bounded in $L^{2}(0, T;X^*)$. Therefore, we take a subsequence $\{ n_j \}\subset \{ n \}$ such that for some $\tilde{u}\in W^{1, 2}(0, T; X^*)\cap V(T)$, $\tilde{u}_{nj} \to \tilde{u}$ strongly in $L^2(0, T; H)$, weakly in $W^{1,2}(0, T; X^*) \cap L^2(0, T; X)$, weakly-* in $L^{\infty}(0, T;H)$, and $\tilde{u}_{nj}(\cdot, x)\to \tilde
{u}(\cdot, x)$ in $L^2(0, T)$ at $x=0$, $1$ as $j \to \infty$. 

Now, we consider the limiting process $j\to \infty$ in the following way:
\begin{align}
\label{3-41}
& \int_0^T\langle \tilde{u}_{nt}(t), z(t) \rangle_X dt + \int_0^T \frac{1}{s^2_{n}(t)} \biggl(\int_0^1 \tilde{u}_{ny}(t) z_y(t) dy\biggr)dt \nonumber \\
& +  \int_0^T \frac{a_0}{s_{n}(t)} \sigma(\tilde{u}_{n}(t, 1)) (\sigma(\tilde{u}_{n}(t, 1))-\alpha s_{n}(t))z(t, 1) dt \nonumber \\
& -\int_0^T\frac{1}{s_{n}(t)}\beta(b(t)-\gamma \tilde{u}_{n}(t, 0))z(t, 0) dt \nonumber \\
& = \int_0^T \int_0^1\frac{y s_{nt}(t)}{s_{n}(t)} \tilde{u}_{ny}(t) z(t) dy dt \mbox{ for }z\in V(T).
\tag{3.45}
\end{align}
Note that by $s_{nj}\to s$ in  $W^{1,2}(0, T)$ as $j\to \infty$, it holds that $s_{nj}\to s$ in $C([0, T])$ as $j\to \infty$. From the convergence of $\tilde{u}_{nj}$ and $s_{nj}$, (\ref{3-32}) and  $a \leq s_{nj}\leq l$ on $[0, T]$, it is clear that 
\begin{align*}
& \int_0^T\langle \tilde{u}_{njt}(t), z(t) \rangle_X dt \to \int_0^T\langle \tilde{u}_{t}(t), z(t) \rangle_X dt, \\
& \int_0^T \frac{a_0}{s_{nj}(t)} \sigma(\tilde{u}_{nj}(t, 1)) (\sigma(\tilde{u}_{nj}(t, 1))-\alpha  s_{nj}(t))z(t, 1) dt \\
& \hspace{5cm}\to \int_0^T \frac{a_0}{s(t)} \sigma(\tilde{u}(t, 1)) (\sigma(\tilde{u}(t, 1))-\alpha s(t))z(t, 1) dt
\end{align*}
and 
\begin{align*}
& \int_0^T\frac{1}{s_{nj}(t)}\beta(b(t)-\gamma \tilde{u}_{nj}(t, 0))z(t, 0) dt 
\to \int_0^T\frac{1}{s(t)}\beta(b(t)-\gamma \tilde{u}(t, 0))z(t, 0) dt \mbox{ as }j\to \infty.
\end{align*}
For the second term in the left-hand side of (\ref{3-41}), it follows that  
\begin{align*}
& \biggl| \int_0^T \frac{1}{s^2_{nj}(t)} \biggl(\int_0^1 \tilde{u}_{njy}(t) z_y(t) dy\biggr)dt -\int_0^T \frac{1}{s^2(t)} \biggl(\int_0^1 \tilde{u}_{y}(t) z_y(t) dy\biggr)dt \biggr| \\
\leq & \int_0^T \biggl | \frac{1}{s^2_{nj}(t)}- \frac{1}{s^2(t)} \biggr|  |\tilde{u}_{njy}(t)|_H |z_y(t)|_H dt \\ 
& + \biggl| \int_0^T \frac{1}{s^2(t)} \biggl(\int_0^1 (\tilde{u}_{njy}(t)- \tilde{u}_y(t)) z_y(t) dy\biggr)dt \biggr| \\
\leq & \frac{2l}{a^2}|s_{nj}-s|_{C([0, T])} |\tilde{u}_{njy}|_{L^2(0, T;H)}|z_y|_{L^2(0, T; H)} \\
& + \biggl | \int_0^T 
\biggl( \tilde{u}_{njy}(t)-\tilde{u}_y(t), \frac{1}{s^2(t)} z_y(t) \biggr)_H dt \biggr| 
\end{align*}
Hence, we observe that 
\begin{align*}
\int_0^T \frac{1}{s^2_{nj}(t)} \biggl(\int_0^1 \tilde{u}_{njy}(t) z_y(t) dy\biggr)dt \to \int_0^T \frac{1}{s^2(t)} \biggl(\int_0^1 \tilde{u}_{y}(t) z_y(t) dy\biggr)dt \mbox{ as }j\to \infty. 
\end{align*}
Also, the right-hand side of (\ref{3-41}) is as follows:
\begin{align*}
& \biggl| \int_0^T \int_0^1\frac{y s_{njt}(t)}{s_{nj}(t)} \tilde{u}_{njy}(t) z(t) dy dt -\int_0^T \int_0^1\frac{y s_{t}(t)}{s(t)} \tilde{u}_{y}(t) z(t) dy dt \biggr| \\
\leq &  \int_0^T \biggl| \frac{s_{njt}(t)}{s_{nj}(t)}-\frac{s_{t}(t)}{s(t)}\biggr| |\tilde{u}_{njy}(t)|_H |z(t)|_H dt  
+ \biggl| \int_0^T \biggl( \tilde{u}_{njy}(t)-\tilde{u}_y(t), \frac{ys_t(t)}{s(t)} z(t) \biggr)_H dt \biggr|   \\
\leq & \frac{1}{a} |s_{njt}-s_t|_{L^2(0, T)}|\tilde{u}_{njy}|_{L^2(0, T;H)}|z|_{L^{\infty}(0, T;H)} \\ 
& + \frac{1}{a^2}|s_{nj}-s|_{C([0, T])}|s_t|_{L^2(0, T)}|\tilde{u}_{njy}|_{L^2(0, T;H)}|z|_{L^{\infty}(0, T;H)} \\
& + \biggl| \int_0^T \biggl( \tilde{u}_{njy}(t)-\tilde{u}_y(t), \frac{ys_t(t)}{s(t)} z(t) \biggr)_H dt \biggr|  
\end{align*}
From this, we have that 
\begin{align*}
\int_0^T \int_0^1\frac{y s_{njt}(t)}{s_{nj}(t)} \tilde{u}_{njy}(t) z(t) dy dt \to \int_0^T \int_0^1\frac{y s_{t}(t)}{s(t)} \tilde{u}_{y}(t) z(t) dy dt 
\mbox{ as }j\to \infty. 
\end{align*}
Finally, by letting $j \to \infty$ we see that $\tilde{u}$ is a solution of $(\mbox{AP})(\tilde{u}_0, s, b)$ on $[0, T]$. Uniqueness is proved by the same argument of the proof of Lemma \ref{lem7}.
\end{proof}

%%%%%%%%%%%%%%%%%%%%%%%%%%%%%%%%%%%%%%%%%%%%%%%%%%%%%%%%%%%%%%%%%%%%%%%%%%%%%%%%%%%%%%%%%%%%
%
%
% subsection 3. 2// proof of Theorem 2.4
%
%5
%%%%%%%%%%%%%%%%%%%%%%%%%%%%%%%%%%%%%%%%%%%%%%%%%%%%%%%%%%%%%%%%%%%%%%%%%%%%%%%%%%%%%%%%%%%%%
\section{Proof of Theorem \ref{t2}}
\label{LE}
In this section, using the results obtained in Section \ref{AP}, we establish the existence of a locally-in-time solution $(\mbox{PC})(\tilde{u}_0, s_0, b)$. Throughout of this section, we assume (A1)-(A3). First, 
for $T>0$, $l>0$ and $a>0$ such that $a<s_0<l$ we set
\begin{align*}
& M(T, a, l):=\{ s\in W^{1,2}(0, T) | a \leq s \leq l \mbox{ on } [0, T], s(0)=s_0\}.
\end{align*}
Also, for given $s\in M(T, a, l)$, we define two solution mappings as follows: $\Psi : M(T, a, l) \to  W^{1,2}(0, T; X^*)\cap L^{\infty}(0, T; H)\cap  L^2(0, T; X)$ by
$\Psi (s)=\tilde{u}$, where $\tilde{u}$ is a unique solution of $(\mbox{AP}1)(\tilde{u}_0, s, b)$ on $[0, T]$ and $\Gamma_{T}: M(T, a, l) \to W^{1,2}(0, T)$ by
$\Gamma_{T}(s)=s_0 + \int_0^t a_0 (\sigma(\Psi(s)(\tau, 1)) -\alpha s(\tau)) d\tau$ for $t\in [0, T]$. Moreover, for any $K>0$ we put
\begin{align*}
& M_K(T):=\{s\in M(T, a, l) | \ |s|_{W^{1,2}(0, T)}\leq K \}.
\end{align*}

%%%%%%%%%%%%%%%%%%%%%%%%%%%%%%%%%%%%%%%%%%%%%%%%%%%%%%%%%%%%%%%%%%%%%%%%%
%
%
% Lemmas 4.1 // 
%
%
%%%%%%%%%%%%%%%%%%%%%%%%%%%%%%%%%%%%%%%%%%%%%%%%%%%%%%%%%%%%%%%%%%%%%%%%%

Now, we show that for some $T>0$, $\Gamma_{T}$ is a contraction mapping on the closed set of $M_K(T)$ for any $K>0$.
\begin{lemma}
\label{lem10}
Let $K>0$. Then, there exists a positive constant $T^*\leq T$ such that 
the mapping $\Gamma_{T^*}$ is a contraction on the closed set $M_K(T^*)$ in $W^{1,2}(0, T^*)$.
\end{lemma}

\begin{proof}
For $T>0$, $a>0$ and $l>0$ such that $a<s_0<l$, let $s\in M(T, a, l)$ and $\tilde{u}=\Psi(s)$. First, we note that it holds 
\begin{align}
\label{4-1}
& |\Psi(s)|_{W^{1,2}(0, T; X^*)} + |\Psi(s)|_{L^{\infty}(0, T;H)} + |\Psi(s)|_{L^2(0, T;X)} \leq C \mbox{ for }s\in M_K(T),
\tag{4.1}
\end{align}
where $C=C(T,\tilde{u}_0, K, l, b^*, \beta,s_0)$ is a positive constant depending on $T$, $\tilde{u}_0$, $K$, $l$, $b^*$, $\beta$ and $s_0$.

First we show that there exists $T_0\leq T$ such that $\Gamma_{T_0}:M_K(T_0)\to M_K(T_0)$ is well-defined. Let $K>0$ and $s\in M_K(T)$. By the definition of $\sigma$ and $\Psi(s)=\tilde{u}$ is a solution of $(\mbox{AP}1)(\tilde{u}_0, s, b)$, we observe that 
\begin{align}
\label{4-2}
\Gamma_{T}(s)(t) &= s_0 + \int_0^t a_0 (\sigma(\Psi(s)(\tau, 1)) -\alpha s(\tau)) d\tau \nonumber \\ 
& \geq s_0 -a_0\alpha lt\mbox{ for }t\in [0, T].
\tag{4.2}
\end{align}
Also, by (\ref{1d}) and (\ref{4-1}), it holds that
%we see that $|\tilde{u}(t, 1)| \leq \sqrt{C_e}|\tilde{u}(t)|_{X} \leq \sqrt{C_e}C$ for a.e. $t\in (0, T)$. 
%From this result and the inequality $\sigma(r)\leq r$ for $r\in \mathbb{R}$ we have that 
\begin{align}
\label{4-4}
& \Gamma_{T}(s)(t) \leq s_0 +a_0T^{1/2} (\int_0^t |\tilde{u}(\tau, 1)|^2d\tau )^{1/2} \nonumber \\
\leq & s_0 + a_0T^{1/2}(C_e |\tilde{u}|_{L^{\infty}(0, t;H)} \int_0^t |\tilde{u}(\tau)|_{X})^{1/2} \leq s_0 + a_0 T^{1/2}(T^{1/4}C^{1/2}_e C),  \nonumber \\ 
& \int_0^t |\Gamma_T(s)(\tau)|^2 d\tau \leq 2s^2_0 T + 4a^2_0 T \int_0^t (|\tilde{u}(\tau, 1)|^2+(\alpha s(t))^2) d\tau \nonumber \\
\leq & 2s^2_0T + 4a^2_0 T \biggl (C_e |\tilde{u}|_{L^{\infty}(0, t;H)}\int_0^t |\tilde{u}(\tau)|_{X} + (\alpha l)^2 T\biggr) \nonumber \\
\leq & 2s^2_0T + 4 a^2_0T (T^{1/2} C_e C^2 + (\alpha l)^2 T),
\tag{4.3}
\end{align}
and
\begin{align}
\label{4-5}
& \int_0^t |\Gamma'_T(s)(\tau)|^2d\tau 
\leq a^2_0\int_0^t|\sigma(\Psi(s)(\tau, 1))-\alpha s(\tau)|^2 d\tau 
\leq 2a^2_0(T^{1/2} C_e C^2 + (\alpha l)^2 T),
\tag{4.4}
\end{align}
where $C$ is the same positive constant as in (\ref{4-1}). 
Therefore,  by (\ref{4-2})-(\ref{4-5}) we see that there exists $T_0\leq T$ such that $\Gamma_{T_0}(s)\in M_K(T_0)$.

Next, for $s_1$ and $s_2\in M_K(T_0)$, let $\tilde{u}_1=\Psi(s_1)$ and $\tilde{u}_2=\Psi(s_2)$ and set $\tilde{u}=\tilde{u}_1-\tilde{u}_2$, $s=s_1-s_2$.
Then, it holds that
\begin{align}
\label{4-6}
& \langle \tilde{u}_t, z \rangle_X + \int_0^1 (\frac{1}{s^2_1} \tilde{u}_{1y}-\frac{1}{s^2_2} \tilde{u}_{2y})z_y dy \nonumber \\
& +  a_0 \biggl( \frac{1}{s_1} \sigma(\tilde{u}_1(\cdot, 1))(\sigma(\tilde{u}_1(\cdot, 1))-\alpha s_1)-\frac{1}{s_2}\sigma(\tilde{u}_2(\cdot , 1))(\sigma(\tilde{u}_2(\cdot, 1))-\alpha s_2) \biggr)z(1)\nonumber \\
& -\biggl (\frac{1}{s_1}\beta (b(t)-\gamma \tilde{u}_1(\cdot, 0))-\frac{1}{s_2}\beta (b(t)-\gamma \tilde{u}_2(\cdot, 0)) \biggr)z(0) \nonumber \\
& = \int_0^1(\frac{y s_{1t}}{s_1} \tilde{u}_{1y}-\frac{y s_{2t}}{s_2} \tilde{u}_{2y})z dy
\mbox{ for }z\in X\mbox{ a.e. on }[0, T].
\tag{4.5}
\end{align}
By taking $z=\tilde{u}$ in (\ref{4-6}) we have 
\begin{align}
\label{4-7}
& \frac{1}{2}\frac{d}{dt}|\tilde{u}|^2_H + \int_0^1 (\frac{1}{s^2_1} \tilde{u}_{1y}-\frac{1}{s^2_2} \tilde{u}_{2y})\tilde{u}_y dy \nonumber \\
& +  a_0 \biggl( \frac{1}{s_1} \sigma(\tilde{u}_1(\cdot, 1))(\sigma(\tilde{u}_1(\cdot, 1))-\alpha s_1))-\frac{1}{s_2}\sigma(\tilde{u}_2(\cdot, 1))(\sigma(\tilde{u}_2(\cdot, 1))-\alpha s_2) \biggr)\tilde{u}(\cdot, 1)\nonumber \\
& -\biggl (\frac{1}{s_1}\beta (b(t)-\gamma \tilde{u}_1(\cdot, 0))-\frac{1}{s_2}\beta (b(t)-\gamma \tilde{u}_2(\cdot, 0) \biggr)\tilde{u}(\cdot, 0) \nonumber \\
& = \int_0^1(\frac{y s_{1t}}{s_1} \tilde{u}_{1y}-\frac{y s_{2t}}{s_2} \tilde{u}_{2y})\tilde{u} dy
\mbox{ a.e. on }[0, T].
\tag{4.6}
\end{align}
For the second term of the left-hand side of (\ref{4-7}), we observe that 
\begin{align}
& \int_0^1 \left(\frac{1}{s^2_1(t)}\tilde{u}_{1y}(t)-\frac{1}{s^2_2(t)}\tilde{u}_{2y}(t)\right) \tilde{u}_y(t) dy \nonumber \\
= &\frac{1}{s^2_1(t)}|\tilde{u}_{y}(t)|^2_{H}+ \int_0^1 \left(\frac{1}{s^2_1(t)}-\frac{1}{s^2_2(t)}\right)\tilde{u}_{2y}(t)\tilde{u}_y (t)dy \nonumber \\
\geq & \frac{1}{s^2_1(t)}|\tilde{u}_{y}(t)|^2_{H} - \frac{2l|s(t)|}{a^3 s_1(t)}|\tilde{u}_{2y}(t)|_{H}|\tilde{u}_y(t)|_{H} \nonumber \\
\geq & \left(1-\frac{\eta}{2}\right)\frac{1}{s^2_1(t)}|\tilde{u}_{y}(t)|^2_{H} -\frac{1}{2\eta}\left(\frac{2l}{a^3}\right)^2|s(t)|^2|\tilde{u}_{2y}(t)|^2_{H},
\tag{4.7}
\end{align}
where $\eta $ is arbitrary positive number. Next, the third term in the left-hand side of (\ref{4-7}) is as follows:
\begin{align*}
&a_0\hspace{-1mm}\left(\frac{\sigma(\tilde{u}_1(t, 1))(\sigma(\tilde{u}_1(t, 1))-\alpha s_1(t))}{s_1(t)}-\frac{\sigma(\tilde{u}_2(t, 1))(\sigma(\tilde{u}_2(t, 1))-\alpha s_2(t))}{s_2(t)}\right)\hspace{-1mm}\tilde{u}(t, 1)\\
=&a_0\biggl[\frac{1}{s_1(t)}\biggl(\sigma(\tilde{u}_1(t, 1))(\sigma(\tilde{u}_1(t, 1))-\alpha s_1(t))-\sigma(\tilde{u}_2(t, 1))(\sigma(\tilde{u}_2(t, 1))-\alpha s_2(t)) \biggr) \\ 
 & + \left(\frac{1}{s_1(t)}-\frac{1}{s_2(t)}\right)\sigma(\tilde{u}_2(t, 1))(\sigma(\tilde{u}_2(t, 1))-\alpha  s_2(t)) \biggr] \tilde{u}(t, 1)\\
=& a_0\biggl[\frac{1}{s_1(t)} (\sigma(\tilde{u}_1(t, 1))-\sigma(\tilde{u}_2(t, 1))) (\sigma(\tilde{u}_1(t, 1))-\alpha s_1(t)) \biggr] \tilde{u}(t, 1)\\
& +\frac{\sigma(\tilde{u}_2(t, 1))}{s_1(t)} \biggl (\sigma(\tilde{u}_1(t, 1))-\alpha s_1(t)-(\sigma(\tilde{u}_2(t, 1))-\alpha s_2(t)) \biggr) \tilde{u}(t, 1)\\ 
& + \left(\frac{1}{s_1(t)}-\frac{1}{s_2(t)}\right)\sigma(\tilde{u}_2(t, 1))(\sigma(\tilde{u}_2(t, 1))-\alpha  s_2(t)) \tilde{u}(t, 1)\\
:= & I_{1} + I_{2} + I_{3}.
\end{align*}
By the monotonicity of $\sigma(r)$ and (\ref{1d}) we have that 
\begin{align}
I_1 \geq -\frac{a_0}{s_1(t)}(|\tilde{u}_1(t, 1)|+\alpha l)|C_e |\tilde{u}(t)|_X |\tilde{u}(t)|_H
\tag{4.8}
\end{align}
and 
\begin{align}
I_{2} & = \frac{a_0}{s_1(t)} \sigma(\tilde{u}_2(t, 1)) \biggl (\sigma(\tilde{u}_1(t, 1))-\alpha s_1(t)-(\sigma(\tilde{u}_2(t, 1))-\alpha s_2(t)) \biggr) \tilde{u}(t, 1) \nonumber \\
& \geq -\frac{a_0 |\tilde{u}_2(t, 1)|}{s_1(t)}  \alpha|s(t)||\tilde{u}(t, 1)| \nonumber \\ 
& \geq - \frac{C_e (a_0 \alpha \tilde{u}_2(t, 1))^2}{2s^2_1(t)} |\tilde{u}(t)|_{X}|\tilde{u}(t)|_{H}- \frac{1}{2}|s(t)|^2.
\tag{4.9}
\end{align}
Also, using the fact that $\sigma(r)\leq |r|$ for $r\in \mathbb{R}$ and  (\ref{1d}), we have the following estimate:
\begin{align}
%\label{4-7}
|I_{3}| & \leq  \left (\frac{|s(t)|}{s_1(t) s_2(t)} \right) a_0 \sigma(\tilde{u}_2(t, 1)) (\sigma(\tilde{u}_2(t, 1)) +\alpha l) \tilde{u}(t, 1) \nonumber \\
& \leq \frac{C_e (a_0\tilde{u}_2(t, 1))^2}{2a^2s^2_1(t)}|\tilde{u}(t)|_{X}|\tilde{u}(t)|_{H} + \frac{1}{2} \tilde{u}^2_2(t, 1)|s(t)|^2 \nonumber \\
& + \frac{C_e(a_0 \alpha l \tilde{u}_2(t, 1))^2}{2a^2s^2_1(t)} |\tilde{u}(t)|_{X}|\tilde{u}(t)|_{H} + \frac{1}{2}|s(t)|^2.
\tag{4.10}
\end{align}
Here, we put $L^{(1)}_{s_1}(t)=a_0 (|\tilde{u}_1(t, 1)|+\alpha l)C_e$ and $L^{(1)}_{s_2}(t) = C_e(a_0 \alpha \tilde{u}_2(t, 1))^2 /2 + C_e (a_0\tilde{u}_2(t, 1))^2/2a^2 + C_e(a_0 \alpha l \tilde{u}_2(t, 1))^2/2a^2$. Next, by using (\ref{1d}) and (A3), we have 
\begin{align}
& -\beta \left (\frac{1}{s_1(t)}-\frac{1}{s_2(t)} \right) b(t) \tilde{u}(t, 0) + \beta \gamma \left(\frac{1}{s_1(t)}\tilde{u}_1(t, 0)-\frac{1}{s_2(t)} \tilde{u}_2(t, 0) \right)\tilde{u}(t, 0) \nonumber \\
= & -\beta \left (\frac{1}{s_1(t)}-\frac{1}{s_2(t)} \right) b(t) \tilde{u}(t, 0) + \beta \gamma \frac{1}{s_1(t)}|\tilde{u}(t, 0)|^2 \nonumber \\
& + \beta \gamma \left(\frac{1}{s_1(t)}-\frac{1}{s_2(t)}\right)\tilde{u}_2(t, 0)\tilde{u}(t, 0) \nonumber \\
\geq & - \left (\frac{(\beta b^*)^2C_e}{2a^2s^2_1(t)} + \frac{(\beta \gamma |\tilde{u}_2(t, 0)|)^2 C_e}{2a^2 s^2_1(t)} \right)|\tilde{u}(t)|_{X}|\tilde{u}(t)|_{H} -|s(t)|^2 \mbox{ for }t\in [0, T_0].
%\label{4-8}
\tag{4.11}
\end{align}
For the right-hand side of (\ref{4-7}), we separate as follows:
\begin{align*}
& \int_0^1 \left(\frac{y s_{1t}(t)}{s_1(t)} \tilde{u}_{1y}(t)-\frac{y s_{2t}(t)}{s_2(t)} \tilde{u}_{2y}(t)\right)\tilde{u}(t)dy\\
=& \int_0^1 \frac{y s_{1t}(t)}{s_1(t)} \tilde{u}_y(t)\tilde{u}(t)dy+\int_0^1 \frac{ys_t(t)}{s_1(t)} \tilde{u}_{2y}(t)\tilde{u}(t)dy \\
& +\int_0^1 \left(\frac{1}{s_1(t)} -\frac{1}{s_2(t)} \right)ys_{2t}(t)\tilde{u}_{2y}(t)\tilde{u}(t)dy \\
:=& I_{4}+I_{5} + I_{6}.
\end{align*}
Then, the three terms are estimated in the following way:
\begin{align*}
I_{4} &\leq \frac{\eta }{2s^2_1(t)}|\tilde{u}_y(t)|^2_{H} + \frac{1}{2\eta }|s_{1t}(t)|^2|\tilde{u}(t)|^2_{H}, \\
I_{5} &\leq \frac{1}{2a}\biggl(|s_t(t)|^2+|\tilde{u}_{2y}(t)|^2_{H}|\tilde{u}(t)|^2_{H}\biggr), \\
I_{6} &\leq \frac{1}{2a^2}\biggl(|s(t)|^2|\tilde{u}_{2y}(t)|^2_{H}+|s_{2t}(t)|^2|\tilde{u}(t)|^2_{H}\biggr).
\end{align*}
%where $\eta$ is an arbitrary positive constant; here $C_6$, $C_7$, and $C_8$ are positive constants depending on the choice of $\eta$.
From (\ref{4-7}) and all estiamtes we derive the following inequality: 
\begin{align}
\label{4-9}
& \frac{1}{2}\frac{d}{dt}|\tilde{u}(t)|^2_{H}+(1-\eta)\frac{1}{s^2_1(t)}|\tilde{u}_y(t)|^2_{H} \nonumber \\
\leq & \frac{1}{s_1(t)} L^{(1)}_{s_1}(t) |\tilde{u}(t)|_{X}|\tilde{u}(t)|_{H} \nonumber \\
& + \frac{1}{s^2_1(t)}\biggl(L^{(1)}_{s_2}(t) + \frac{(\beta \gamma |\tilde{u}_2(t, 0)|)^2C_e}{2a^2}+\frac{(\beta b^*)^2C_e}{2a^2} \biggr) |\tilde{u}(t)|_{X}|\tilde{u}(t)|_{H} \nonumber \\
& + \left(\frac{1}{2\eta }|s_{1t}(t)|^2 +  \frac{1}{2a}|\tilde{u}_{2y}(t)|^2_{H} + \frac{1}{2a^2}|s_{2t}(t)|^2 \right)|\tilde{u}(t)|^2_{H} \nonumber \\
& + \left( \frac{1}{2a^2}|\tilde{u}_{2y}(t)|^2_{H} + \frac{1}{2\eta}\left(\frac{2l}{a^3}\right)^2|\tilde{u}_{2y}(t)|^2_{H} + \frac{1}{2}\tilde{u}^2_2(t, 1)+2\right)|s(t)|^2 + \frac{1}{2a}|s_t(t)|^2.
\tag{4.12}
\end{align}
We put $C_5(t)=((\beta \gamma |\tilde{u}_2(t, 0)|)^2C_e)/2a^2 + ((\beta b^*)^2C_e)/2a^2$. By Young's inequality we have that 
%Here, by (\ref{1d}) and (\ref{4-1}), it holds that
%\begin{align}
%|\tilde{u}_i(t, 0)|^2 \leq C_e (|\tilde{u}_{iy}(t)|_{H}|\tilde{u}_i(t)|_{H} + |\tilde{u}_i(t)|^2_{H})\leq 2C_eC^2 \mbox{ for }t\in [0, T_0],
%\label{4-10}
%\end{align}
%where $C$ is the same constant as in (\ref{4-1}). Then, by (\ref{4-10}) we note that $\{ L^{(1)}_{s_2} | s_2\in M_k(T)\}$ is bounded in %$L^{\infty}(0, T_0)$. Also, by putting $C_5=(\beta (b^*+2\gamma C_eC^2))^2C_e$ and Young's inequality we have that 
\begin{align*}
&  \frac{1}{s_1(t)} L^{(1)}_{s_1}(t)|\tilde{u}(t)|_{X}|\tilde{u}(t)|_{H}\\
\leq &  \frac{1}{s_1(t)} L^{(1)}_{s_1}(t) \biggl(|\tilde{u}_y(t)|_{H}|\tilde{u}(t)|_{H} + |\tilde{u}(t)|^2_{H}\biggr)\\
\leq & \frac{\eta}{2s^2_1(t)}|\tilde{u}_y(t)|^2_{H} + \biggl ( \frac{(L^{(1)}_{s_1}(t))^2}{2\eta} +  \frac{1}{a}L^{(1)}_{s_1}(t)\biggr)|\tilde{u}(t)|^2_{H},
\end{align*}
and
\begin{align*}
& (L^{(1)}_{s_2}(t) + C_5(t))\frac{1}{s^2_1(t)}|\tilde{u}(t)|_{X}|\tilde{u}(t)|_{H} \\
\leq & (L^{(1)}_{s_2}(t) +C_5(t))\frac{1}{s^2_1(t)}(|\tilde{u}_y(t)|_{H}|\tilde{u}(t)|_{H} + |\tilde{u}(t)|^2_{H})\\
\leq & \frac{1}{s^2_1(t)}\frac{\eta }{2}|\tilde{u}_y(t)|^2_{H} + \frac{1}{a^2} \left(\frac{(L^{(1)}_{s_2}(t) +C_5(t))^2}{2\eta } + (L^{(1)}_{s_2}(t) +C_5(t))\right)|\tilde{u}(t)|^2_{H}. 
\end{align*}
Hence, by applying these results to (\ref{4-9}) and taking a suitable $\eta=\eta_0$, we obtain
\begin{align}
\label{4-11}
& \frac{1}{2}\frac{d}{dt}|\tilde{u}(t)|^2_{H}+\frac{1}{2}\frac{1}{s^2_1(t)}|\tilde{u}_y(t)|^2_{H} \nonumber \\
\leq &  \biggl ( \frac{(L^{(1)}_{s_1}(t))^2}{2\eta_0} +  \frac{1}{a}L^{(1)}_{s_1}(t)\biggr)|\tilde{u}(t)|^2_{H} \nonumber \\
& + \frac{1}{a^2} \left(\frac{(L^{(1)}_{s_2}(t) +C_5(t))^2}{2\eta_0 } + (L^{(1)}_{s_2}(t) +C_5(t))\right)|\tilde{u}(t)|^2_{H}. \nonumber \\
& + \left(\frac{1}{2\eta_0 }|s_{1t}(t)|^2 +  \frac{1}{2a}|\tilde{u}_{2y}(t)|^2_{H} + \frac{1}{2a^2}|s_{2t}(t)|^2 \right)|\tilde{u}(t)|^2_{H} \nonumber \\
& + \left(\frac{1}{2a^2}|\tilde{u}_{2y}(t)|^2_H + \frac{1}{2\eta_0}\left(\frac{2l}{a^3}\right)^2|\tilde{u}_{2y}(t)|^2_H+ \frac{1}{2}\tilde{u}^2_2(t, 1)+2\right)|s(t)|^2 + \frac{1}{2a}|s_t(t)|^2.
\tag{4.13}
\end{align}
Now, we put the summation of all coefficients of $|\tilde{u}(t)|^2_{H}$ by $L^{(2)}_{s}(t)$ for $t\in [0, T_0]$ and 
$L^{(3)}_{s_2}(t)=|\tilde{u}_{2y}(t)|^2_{H}/2a^2 + (4l^2|\tilde{u}_{2y}(t)|^2_{H})/2\eta_0a^6+\tilde{u}^2_2(t, 1)/2+2$. Then, we have
\begin{align}
\label{4-12} 
& \frac{1}{2}\frac{d}{dt}|\tilde{u}(t)|^2_{H}+\frac{1}{2}\frac{1}{s^2_1(t)}|\tilde{u}_y(\tau)|^2_{H} \nonumber \\
\leq & L^{(2)}_{s}(t)|\tilde{u}(t)|^2_{H} + L^{(3)}_{s_2}(t)|s(t)|^2 + \frac{1}{2a}|s_t(t)|^2 \mbox{ for }t\in [0, T_0].
\tag{4.14}
\end{align}
By (\ref{1d}) and (\ref{4-1}) we note that $\tilde{u}^2_i(\cdot, 1), \tilde{u}^4_i(\cdot, 1)\in L^1(0, T_0)$ for $i=1, 2$. From this  and the fact that $s_i\in M_K(T_0)$ for $i=1, 2$, we see that $L^{(2)}_s\in L^1(0, T_0)$ and $L^{(3)}_{s_2}\in L^1(0, T_0)$.
Also, it holds that 
\begin{align*}
L^{(3)}_{s_2}(t)|s(t)|^2 \leq L^{(3)}_{s_2}(t) T_0 |s_t|^2_{L^2(0, t)} \mbox{ for }t\in [0, T_0].
\end{align*}
Therefore, Gronwall's inequality guarantees that
\begin{align}
\label{4-13}
& \frac{1}{2}|\tilde{u}(t)|^2_{H}+\frac{1}{2}\frac{1}{l^2}\int_0^t|\tilde{u}_y(\tau)|^2_{H} d\tau \nonumber \\
\leq & \left[ \left(2|L^{(3)}_{s_2}|_{L^1(0, T_0)}T_0+\frac{1}{a} \right)|s_t|^2_{L^{2}(0, t)} \right]  e^{2\int_0^t L^{(2)}_{s}(\tau)d\tau} \mbox{ for }t\in [0, T_0].
\tag{4.15}
\end{align}
By using (\ref{4-13}) we show that there exists $T^*\leq T_0$ such that $\Gamma_{T^*}$ is a contraction mapping on the closed subset of $M_K(T^*)$.
To do so, from the subtraction of the time derivatives of $\Gamma_{T_0}(s_1)$ and $\Gamma_{T_0}(s_2)$ and relying on (\ref{1d}) and (\ref{4-13}), we have for $T_1\leq T_0$ the following estimate:
\begin{align}
\label{4-14}
& |(\Gamma_{T_1}(s_1))_t-(\Gamma_{T_1}(s_2))_t|_{L^2(0, T_1)} \nonumber \\
\leq & a_0 \biggl(|\sigma(\tilde{u}_1(\cdot, 1))-\sigma(\tilde{u}_2(\cdot, 1))|_{L^2(0, T_1)} + \alpha |s_1(t)-s_2(t)|_{L^2(0, T_1)}\biggr) \nonumber \\
\leq & a_0\sqrt{C_e}\biggl(\int_0^{T_1}(|\tilde{u}_y(t)|_{H} |\tilde{u}(t)|_H + |\tilde{u}(t)|^2_{H})dt \biggr)^{1/2} + a_0 \alpha T_1 |s|_{W^{1,2}(0, T_1)}\nonumber \\
\leq & a_0 \sqrt{C_e} \left(|\tilde{u}|^{\frac{1}{2}}_{L^{\infty}(0, T_1;H)}\left(\int_0^{T_1}|\tilde{u}_y(t)|_{H} dt \right)^{\frac{1}{2}} + \sqrt{T_1} |\tilde{u}|_{L^{\infty}(0, T_1; H)} \right) \nonumber \\
& + a_0 \alpha T_1 |s|_{W^{1,2}(0, T_1)}.
\tag{4.16}
\end{align}
Using (\ref{4-13}) and (\ref{4-14}), we obtain
\begin{align}
\label{4-14-1}
& |\Gamma_{T_1}(s_1)-\Gamma_{T_1}(s_2)|_{L^2(0, T_1)} \nonumber \\
\leq & T_1C_6 \biggl (T^{\frac{1}{4}}_1|s|_{W^{1,2}(0, T_1)}+ \sqrt{T_1}|s|_{W^{1,2}(0, T_1)} + T_1|s|_{W^{1,2}(0, T_1)}\biggr),
% \leq T^{1/4} |s_t|_{L^4(0, T_1)} + T^{1/2} |s_t|_{L^4(0, T_1)} + T_1 |s_t|_{L^{4}(0, T_1)}
\tag{4.17}
\end{align}
where $C_6$ is a positive constant obtained by (\ref{4-13}). Therefore, by (\ref{4-14}) and (\ref{4-14-1}) we see that there exists $T^*\leq T_0$ such that $\Gamma_{T^*}$ is a contraction mapping on a closed subset of $M_K(T^*)$.
\end{proof}

From Lemma \ref{lem10}, by applying Banach's fixed point theorem, there exists $s\in M_K(T^*)$, where $T^*$ is the same as in Lemma \ref{lem10} such that
$\Gamma_{T^*}(s)=s$. This implies that $(\mbox{PC})(\tilde{u}_0, s_0, b)$ has a unique solution $(s, \tilde{u})$ on $[0, T^*]$. 
%Note that by (2.5), $s_t(t) \geq 0$ for $t\in [0, T^*]$. Thus, we can prove Theorem \ref{t2}. Moreover, this shows that 
%by the change of variables (\ref{2-7})
%a pair of the function $(s, u)$ is a solution of $(\mbox{P})(u_0, s_0, b)$ on $[0, T^*]$.

At the end of this section, we show the boundedness of  a solution $\tilde{u}$ to $(\mbox{PC})(\tilde{u}_0, s_0, b)$ which completes Theorem \ref{t2}. 
By (S2) it holds that 
\begin{align}
\label{4-15-0}
& \langle \tilde{u}_t, z \rangle_X + \int_0^1 \frac{1}{s^2} \tilde{u}_{y} z_y dy + \frac{1}{s} \sigma(\tilde{u}(\cdot, 1))s_t z(1)\nonumber \\
&-\frac{1}{s}\beta (b(\cdot)-\gamma \tilde{u}(\cdot, 0))z(0) = \int_0^1\frac{y s_{t}}{s} \tilde{u}_{y}z dy \mbox{ for }z\in X\mbox{ a.e. on }[0, T].
\tag{4.18}
\end{align}

First, we note that the solution $\tilde{u}$ of $(\mbox{PC})(\tilde{u}_0, s_0, b)$ is non-negative. 
%This result is proved by the same argument of the proof of Lemma \ref{lem9} so that we omit the proof. 
%Next, 
%that the free boundary $s$ has a positive lower bound of and $u$ is non-negative. %which completes Theorem \ref{t1}.
%%%%%%%%%%%%%%%%%%%%%%%%%%%%%%%%%%%%%%%%%%%%%%%%%%%%%%%%%%%%%%%%%%%%%%%%%%%%%%
%
%
% Lemma 4.2 //  
%
%
%%%%%%%%%%%%%%%%%%%%%%%%%%%%%%%%%%%%%%%%%%%%%%%%%%%%%%%%%%%%%%%%%%%%%%%%%%%%%%%%
%\begin{comment}
\begin{lemma}
\label{lem11}
Let $T>0$ and $(s, \tilde{u})$ be a solution of $(\mbox{PC})(\tilde{u}_0, s_0, b)$ on $[0, T]$. Then, $\tilde{u}(t)\geq 0$ on $[0, 1]$ for $t\in [0, T]$.
%there exists a positive constant $s_*(T)$ depending on $T$ and a non-decreasing function $M(T)$ of $T$ such that 
%\begin{align*}
%(i) & \quad s(t) \geq s_*(T) \mbox{ for }t\in [0, T], \\
%(ii) & \quad u(t) \geq 0 \mbox{ for }t\in [0, T],
%\quad |u(t)|^2_{L^2(0, s(t))} + \int_0^{t} |u_z(\tau)|^2_{L^2(0, s(\tau)} + \int_0^t |b(\tau)-\gamma u(\tau, 0)|^2 d\tau \leq M(T) \mbox{ for }t\in [0, T]. 
%\end{align*}
%where $s_*(T)$ is a positive constant depending on $T$. 
\end{lemma}

\begin{proof}
We use the same argument of the proof of Lemma \ref{lem9}. 
%First, we show (i). By (1.4) and the definition of $\sigma$, we see that 
%\begin{align}
%s_t(t) + a_0 \alpha s(t)=a_0\sigma(u(t, s(t))) \geq 0 \mbox{ for a.e. }t\in [0, T]. 
%\label{4-15-0}
%\end{align}
%This implies that 
%\begin{align*}
%\left (s(t) e^{a_0\alpha t} \right)_t= e^{a_0\alpha t}(s_t(t) + a_0\alpha s(t)) \geq 0 \mbox{ for a.e. }t\in [0, T].
%\end{align*}
%Therefore, by integrating over $[0, t]$ for any $t\in [0, T]$ we have $s(t) \geq s_0 e^{-a_0\alpha t} \geq s_0 e^{-a_0\alpha T}$ for $t\in [0, T]%$. Thus, $s_*(T)=s_0 e^{-a_0\alpha T}$ is a desired one. 
%Next, we show that $u(t) \geq 0$ on $[0, s(t)]$ for $t\in [0, T]$. 
Indeed, by (\ref{4-15-0}) we have %By taking $z=-[-\tilde{u}]^+$ in (\ref{4-15-0}) we have 
\begin{align}
\label{4-15}
& \frac{1}{2} \frac{d}{dt}|[-\tilde{u}(t)]^+|^2_{H}+ \frac{1}{s^2(t)} \int_0^1 |[-\tilde{u}(t)]^+_y|^2 dy -\frac{a_0}{s(t)} \sigma(\tilde{u}(\cdot, 1))s_t(t)[-\tilde{u}(t, 1)]^+\nonumber \\
& +\frac{1}{s(t)}\beta (b(t)-\gamma \tilde{u}(\cdot, 0))[-\tilde{u}(t, 0)]^+= -\int_0^1\frac{y s_t(t)}{s(t)} \tilde{u}_y(t)[-\tilde{u}(t)]^+dy
\mbox{ for a.e. }t\in [0, T].
\tag{4.19}
\end{align}
Here, the third term in the left-hand side of (\ref{4-15}) is equal to 0 and the forth term in the left-hand side of (\ref{4-15}) is non-negative. 
%Also, from (1.4) and (\ref{1d}) we have that 
%\begin{align*}
%& -\frac{a_0}{s(t)}\sigma(\tilde{u}(\cdot, 1))(\sigma(\tilde{u}(\cdot, 1)-\alpha s(t))[-\tilde{u}(t, 1)]^+ \\
%& \geq -a_0 \alpha |[-\tilde{u}(t, 1)]^+|^2 \\
%& \geq -a_0 \alpha C_e (|[-\tilde{u}(t, 1)]^+_y|_{H} |[-\tilde{u}(t, 1)]^+|_{H} +|[-\tilde{u}(t, 1)]^+|^2_{H}) \\
%& \geq -\frac{1}{4s^2(t)}|[-\tilde{u}(t, 1)]^+_y|^2_{H} -( l^2(a_0 \alpha C_e)^2 + a_0\alpha C_e)|[-\tilde{u}(t, 1)]^+|^2_{H},
%\end{align*}
Also, by $s_t(t)=a_0(\sigma(\tilde{u}(t, 1))-\alpha s(t))$ we note  that 
\begin{align*}
& \int_0^1\frac{y s_t(t)}{s(t)} [-\tilde{u}(t)]^+_y [-\tilde{u}(t)]^+dy = \frac{s_t(t)}{s(t)}\int_0^1 \frac{1}{2} \biggl( \frac{d}{dy} (y|[-\tilde{u}(t)]^+|^2)-|[-\tilde{u}(t)]^+|^2 \biggr) \\
% \frac{y}{s(t)} (-a_0 \alpha s(t)) [-\tilde{u}(t)]^+_y[-\tilde{u}(t)]^+dy \\ 
%& \leq a_0 \alpha |[-\tilde{u}(t)]^+_y|_{H}|[-\tilde{u}(t)]^+|_{H} \\
& = \frac{s_t(t)}{2s(t)} |[-\tilde{u}(t, 1)]^+|^2- \frac{s_t(t)}{2s(t)} |[-\tilde{u}(t)]^+|^2_H \\
& = \frac{a_0(\sigma(\tilde{u}(t, 1))-\alpha s(t))}{2s(t)} |[-\tilde{u}(t, 1)]^+|^2- \frac{a_0(\sigma(\tilde{u}(t, 1))-\alpha s(t))}{2s(t)} |[-\tilde{u}(t)]^+|^2_H \\
& \leq \frac{a_0\alpha}{2}|[-\tilde{u}(t)]^+|^2_H. %\frac{1}{2s^2(t)}|[-\tilde{u}(t)]^+_y|^2_{H} + \frac{(a_0\alpha l(T))^2}{2}|[-\tilde{u}(t)]^+|^2_{H}, 
\end{align*}
%where $l(T)=\max_{0\leq t\leq T}s(t)$.
Then, from (\ref{4-15}) we have 
\begin{align*}
\frac{1}{2} \frac{d}{dt} |[-\tilde{u}(t)]^+|^2_{H} + \frac{1}{2s(t)} \int_0^1 |[-\tilde{u}(t)]^+_y|^2 dy \leq \frac{a_0\alpha}{2}|[-\tilde{u}(t)]^+|^2_H \mbox{ for a.e. }t\in [0, T].
\end{align*}
Therefore, by Gronwall's inequality and the fact that $\tilde{u}_0 \geq 0$ on $[0, 1]$, we conclude that $\tilde{u}(t) \geq 0$ on $[0, 1]$ for $t\in [0, T]$. 
\end{proof}
%\end{comment}

Next, we show the boundedness of the solution $(s, \tilde{u})$ of $ (\mbox{PC})(\tilde{u}_0, s_0, b)$.
%%%%%%%%%%%%%%%%%%%%%%%%%%%%%%%%%%%%%%%%%%%%%%%%%%%%%%%%%%%%%%%%%%%%%%%%%%%%%%%
%
%
% Lemma 4.3 // 
%
%
%%%%%%%%%%%%%%%%%%%%%%%%%%%%%%%%%%%%%%%%%%%%%%%%%%%%%%%%%%%%%%%%%%%%%%%%%%%%%%%

\begin{lemma}
\label{lem12}
Let $T>0$ and $(s, \tilde{u})$ be a solution of $(\mbox{PC})(\tilde{u}_0, s_0, b)$ on $[0, T]$. Then, it holds that 
\begin{align*}
(i) & \quad s(t) \leq M \mbox{ for }t\in [0, T], \\ %+ \int_0^t |\tilde{u}_y(t)|^2_H dy \leq M \mbox{ for }t\in [0, T], \\ 
(ii) & \quad 0\leq \tilde{u}(t) \leq u^*:=\max \{ \alpha M, \frac{b^*}{\gamma} \} \mbox{ on } [0, 1] \mbox{ for } t\in [0, T],
%where $l(T)=\max_{0\leq t \leq T}|s(t)|$
%\begin{align*}
%(i) & \quad u(t) \geq 0 \mbox{ on } [0, s(t)] \mbox{ for } t\in [0, T], \\
%(ii) & \quad s(t)\leq M_2 \mbox{ for } t\in [0, T], 
\end{align*}
where $M$ is a positive constant which depends on $\beta$, $\gamma$, $\alpha$, $b^*$, $|\tilde{u}_0|_{H}$, $s_0$, $|b_t|_{L^2(0, T)}$ and $|b_t|_{L^1(0, T)}$. 
\end{lemma}

\begin{proof}
First, we prove (i). By taking $z=s(\tilde{u}-\frac{b}{\gamma})$ in (\ref{4-15-0}) it holds that  
\begin{align}
\label{4-19}
& \frac{s(t)}{2} \frac{d}{dt}\biggl |\tilde{u}(t)-\frac{b(t)}{\gamma} \biggr|^2_{H} + \biggl (\frac{b_t(t)}{\gamma}, s(t)(\tilde{u}(t)-\frac{b(t)}{\gamma}) \biggr)_H \nonumber \\
& + \frac{1}{s(t)} \int_0^1 |\tilde{u}_y(t)|^2 dy + \sigma(\tilde{u}(t, 1))s_t(t) \biggl (\tilde{u}(t, 1)-\frac{b(t)}{\gamma} \biggr) \nonumber \\
& - \beta (b(t)-\gamma \tilde{u}(t, 0)) \biggl (\tilde{u}(t, 0)-\frac{b(t)}{\gamma} \biggr)= \int_0^1y s_t(t) \tilde{u}_y(t) \biggl (\tilde{u}(t)-\frac{b(t)}{\gamma} \biggr) dy \nonumber \\
& \mbox{ for a.e. }t\in [0, T].
\tag{4.20}
\end{align}
%From (1.1) and the boundary condition (1.2) and (1.3), we have 
%\begin{align}
%& \frac{1}{2} \frac{d}{dt} |u(t)|^2_{L^2(0, s(t))} + \frac{s_t(t)}{2}|u(t, s(t))|^2 + \int_0^{s(t)} |u_z(t)|^2 dz \nonumber \\
%& - \beta(b(t)-\gamma u(t, 0)) u(t, 0)=0 \mbox{ for a.e. }t\in [0, T].
%\label{4-19-1}
%\end{align}
The second term of the left-hand side of (\ref{4-19}) is follows:
\begin{align}
\label{4-20}
& \biggl | \biggl (\frac{b_t(t)}{\gamma}, s(t)(\tilde{u}(t)-\frac{b(t)}{\gamma}) \biggr)_H \biggr| = \biggl | \frac{b_t(t)}{\gamma} s(t) \int_0^1(\tilde{u}(t, y)-\tilde{u}(t, 0)+\tilde{u}(t, 0)-\frac{b(t)}{\gamma}) dy \biggr| \nonumber \\
\leq & \frac{1}{\gamma}|b_t(t)| |s(t)| \biggl (|\tilde{u}_y(t)|_H + \biggl |\tilde{u}(t, 0)-\frac{b(t)}{\gamma} \biggr| \biggr).
\tag{4.21}
%\leq & \frac{1}{2s(t)}|\tilde{u}_y(t)|^2_H + \frac{1}{2} s^3(t) \biggl (\frac{|b_t(t)|}{\gamma} \biggr)^2+ \frac{\beta \gamma}{2}|\tilde{u}(t, 0)-\frac{b(t)}{\gamma}|^2 + \frac{1}{2\beta \gamma} |s(t)|^2 \biggl (\frac{|b_t(t)|}{\gamma} \biggr)^2
\end{align}
Also, by Lemma \ref{lem11} we note that $\sigma(\tilde{u}(t, 1))s_t(t)=\tilde{u}(t, 1)s_t(t)$.  Moreover, we observe that
\begin{align}
\label{4-21}
& \int_0^1y s_t(t) \tilde{u}_y(t) \biggl (\tilde{u}(t)-\frac{b(t)}{\gamma} \biggr) dy \nonumber \\
& = s_t(t) \int_0^1\frac{1}{2} \biggl( \frac{\partial}{\partial y} \biggl (y(\tilde{u}(t)-\frac{b(t)}{\gamma})^2\biggr)-(\tilde{u}(t)-\frac{b(t)}{\gamma})^2 \biggr)dy \nonumber \\
&=\frac{s_t(t)}{2}\biggl (\tilde{u}(t, 1)-\frac{b(t)}{\gamma}\biggr)^2-\frac{s_t(t)}{2}\int_0^1 \biggl (\tilde{u}(t)-\frac{b(t)}{\gamma}\biggr)^2 dy.
\tag{4.22} 
\end{align}
Hence, by (\ref{4-20}) and (\ref{4-21}) we have 
 \begin{align}
\label{4-22}
& \frac{1}{2} \frac{d}{dt} \biggl (s(t) \biggl |\tilde{u}(t)-\frac{b(t)}{\gamma} \biggr|^2_{H} \biggr) + \frac{1}{s(t)} \int_0^1 |\tilde{u}_y(t)|^2 dy \nonumber \\
& + \tilde{u}(t, 1)s_t(t) \biggl (\tilde{u}(t, 1)-\frac{b(t)}{\gamma} \biggr) -\frac{s_t(t)}{2}\biggl (\tilde{u}(t, 1)-\frac{b(t)}{\gamma}\biggr)^2 + \beta \gamma |\tilde{u}(t, 0)-\frac{b(t)}{\gamma}| ^2 \nonumber \\
\leq  & \frac{1}{\gamma}|b_t(t)||s(t)| \biggl (|\tilde{u}_y(t)|_H + \biggl |\tilde{u}(t, 0)-\frac{b(t)}{\gamma} \biggr| \biggr) \mbox{ for a.e. }t\in [0, T].
\tag{4.23}
\end{align}
Here, the third and forth terms of the left-hand side of (\ref{4-22}) is as follows:
\begin{align}
\label{4-23}
& \tilde{u}(t, 1)s_t(t) \biggl (\tilde{u}(t, 1)-\frac{b(t)}{\gamma} \biggr) -\frac{s_t(t)}{2}\biggl (\tilde{u}(t, 1)-\frac{b(t)}{\gamma}\biggr)^2 \nonumber \\
= & s_t(t)\tilde{u}^2(t, 1)-s_t(t)\tilde{u}(t, 1)\frac{b(t)}{\gamma}-\frac{s_t(t)}{2} \biggl (\tilde{u}^2(t, 1)-2\tilde{u}(t, 1)\frac{b(t)}{\gamma} + \biggl (\frac{b(t)}{\gamma} \biggr)^2 \biggr) \nonumber \\
= & \frac{s_t(t)}{2} \tilde{u}^2(t, 1)-\frac{s_t(t)}{2} \biggl (\frac{b(t)}{\gamma} \biggr)^2.
\tag{4.24}
\end{align}
Since $s_t(t)=a_0(\tilde{u}(t, 1)-\alpha s(t))$, it holds that 
\begin{align}
\label{4-24}
& \frac{s_t(t)}{2} \tilde{u}^2(t, 1) = \frac{\tilde{u}(t, 1)}{2} \left(\frac{|s_t(t)|^2}{a_0} + \alpha s(t) s_t(t)\right) \nonumber \\
&= \frac{1}{2a_0} \tilde{u}(t, 1)|s_t(t)|^2 + \frac{\alpha s(t)}{2} \left(\frac{|s_t(t)|^2}{a_0} + \alpha s(t) s_t(t)\right)  \nonumber \\
& =\frac{1}{2a_0} \tilde{u}(t, 1)|s_t(t)|^2 + \frac{\alpha s(t)}{2} \frac{|s_t(t)|^2}{a_0} + \frac{\alpha^2}{6} \frac{d}{dt}s^3(t).
%\geq - \frac{a_0\alpha}{2} s(t) |u(t, s(t))|^2 \geq - \frac{a_0\alpha}{2} l |u(t, s(t))|^2 
\tag{4.25}
\end{align} 
Note that the first and second terms in the right-hand side of (\ref{4-24}) are non-negative. 
Accordingly, by (\ref{4-23}) and (\ref{4-24}) we have that 
\begin{align}
\label{4-25}
& \tilde{u}(t, 1)s_t(t) \biggl (\tilde{u}(t, 1)-\frac{b(t)}{\gamma} \biggr) -\frac{s_t(t)}{2}\biggl (\tilde{u}(t, 1)-\frac{b(t)}{\gamma}\biggr)^2 \nonumber \\
\geq & \frac{\alpha^2}{6} \frac{d}{dt}s^3(t)-\frac{s_t(t)}{2} \biggl (\frac{b(t)}{\gamma} \biggr)^2 \nonumber \\
=  & \frac{\alpha^2}{6} \frac{d}{dt}s^3(t)-\frac{1}{2\gamma^2} \biggl (\frac{d}{dt}  (s(t)b^2(t))-2s(t)b(t)b_t(t) \biggr).
\tag{4.26}
\end{align}
Also, for the right-hand side of (\ref{4-22}) we have 
\begin{align}
\label{4-26}
& \frac{1}{\gamma}|b_t(t)||s(t)| \biggl (|\tilde{u}_y(t)|_H +|\tilde{u}(t, 0)-\frac{b(t)}{\gamma}| \biggr) \nonumber \\
\leq & \frac{1}{2s(t)}|\tilde{u}_y(t)|^2_H + \frac{1}{2} s^3(t) \biggl (\frac{|b_t(t)|}{\gamma} \biggr)^2+ \frac{\beta \gamma}{2}\biggl |\tilde{u}(t, 0)-\frac{b(t)}{\gamma} \biggr|^2 + \frac{1}{2\beta \gamma} |s(t)|^2 \biggl (\frac{|b_t(t)|}{\gamma} \biggr)^2.
%& - \beta (b(t)-\gamma \tilde{u}(t, 0)) \biggl (\tilde{u}(t, 0) -\frac{b^*}{\gamma} \biggr)= \beta(\gamma \tilde{u}(t, 0)-b^*+b^*-b(t)) \biggl (\tilde{u}(t, 0)-\frac{b^*}{\gamma} \biggr) \nonumber \\
%& = \beta \gamma \biggl |\tilde{u}(t, 0)-\frac{b^*}{\gamma} \biggr|^2 + \beta(b^*-b(t))\biggl (\tilde{u}(t, 0)-\frac{b^*}{\gamma} \biggr) \geq 0.
%& \geq \frac{\beta}{2\gamma l} |b(t)-\gamma u(t, 0)|^2 - \frac{\beta}{2\gamma s_*(T)}|b(t)|^2
%\frac{\beta\gamma}{2} |u(t, 0)|^2-\frac{\beta}{2\gamma} |b(t)|^2
\tag{4.27}
\end{align}
By combining (\ref{4-25}) and (\ref{4-26}) with (\ref{4-22}) we have 
\begin{align}
\label{4-27}
& \frac{1}{2} \frac{d}{dt} \biggl (s(t) \biggl |\tilde{u}(t)-\frac{b(t)}{\gamma} \biggr|^2_{H} \biggr) + \frac{1}{2s(t)} \int_0^1 |\tilde{u}_y(t)|^2 dy  \nonumber \\
& +\frac{\alpha^2}{6}\frac{d}{dt} s^3(t) - \frac{1}{2\gamma^2} \frac{d}{dt}  (s(t)b^2(t)) + \frac{\beta \gamma}{2}|\tilde{u}(t, 0)-\frac{b(t)}{\gamma}|^2  \nonumber \\
\leq & \frac{1}{\gamma^2} s(t)b(t)|b_t(t)| + \frac{1}{2} s^3(t) \biggl (\frac{|b_t(t)|}{\gamma} \biggr)^2 + \frac{1}{2\beta \gamma} |s(t)|^2 \biggl (\frac{|b_t(t)|}{\gamma} \biggr)^2
\mbox{ for a.e. }t\in [0, T]. 
\tag{4.28}
\end{align}
%Here, by the fact that $s_*(T)\leq s(t)$ for $t\in [0,T]$ and (1d), we note that 
 %it holds that 
%\begin{align*}
%& |u(t, s(t))|^2 = |\tilde{u}(t, 0)|^2 \leq C_e(|\tilde{u}_y(t)|_{L^2(0, 1)}|\tilde{u}(t)|_{L^2(0, 1)} + |\tilde{u}(t)|^2_{L^2(0, 1)} )\\
%& \leq C_e \left (|u_z(t)|_{L^2(0, s(t))}|u(t)|_{L^2(0, s(t))} + \frac{1}{s(t)} |u(t)|^2_{L^2(0, s(t))} \right) \\
%& \leq C_e \left (|u_z(t)|_{L^2(0, s(t))}|u(t)|_{L^2(0, s(t))} + \frac{1}{s_*(T)} |u(t)|^2_{L^2(0, s(t))} \right) 
%\end{align*}
Then, by integrating (\ref{4-27}) over $[0, t]$ for $t\in [0, T]$ we obtain that 
\begin{align}
\label{4-28}
& \frac{1}{2}s(t) \biggl |\tilde{u}(t)-\frac{b(t)}{\gamma} \biggr|^2_{H} + \int_0^t \frac{1}{2s(\tau)} |\tilde{u}_y(\tau)|^2_H d\tau + \frac{\alpha^2}{6} s^3(t) - \frac{1}{2\gamma^2} (s(t)b^2(t)) \nonumber \\
\leq &  \frac{1}{2}s_0 \biggl |\tilde{u}_0-\frac{b(0)}{\gamma} \biggr|^2_{H} + \frac{\alpha^2}{6} s^3_0 +\frac{1}{2\gamma^2} \int_0^t s^3(\tau) |b_t(\tau)|^2d\tau \nonumber \\
& +  \frac{b^*}{\gamma^2} \int_0^t s(\tau)|b_t(\tau)|d\tau + \frac{1}{2\beta \gamma^3} \int_0^t |s(\tau)|^2 |b_t(\tau)|^2 d\tau 
\mbox{ for }t\in [0, T].
\tag{4.29}
%& \frac{1}{2} \frac{d}{dt} |u(t)|^2_{L^2(0, s(t)} + \frac{1}{2}\int_0^{s(t)} |u_z(t)|^2 dz + \frac{\beta}{2\gamma} \int_0^t |b(\tau)-\gamma u(\tau, 0)|^2\\
%\leq & \biggl[ \frac{1}{2}\left( \frac{a_0\alpha}{2} l(T) \right)^2 C^2_e+ \frac{a_0\alpha}{2}\frac{C_e}{s_*(T)}l(T) \biggr]|u(t)|^2_{L^2(0, s(t)} + \frac{\beta}{2\gamma} |b(t)|^2 \mbox{ for a.e. }t\in [0, T]. 
\end{align}
Here, by Young's inequality it follows that 
\begin{align}
\label{4-29}
& \frac{\alpha^2}{6} s^3(t) -\frac{1}{2\gamma^2} (s(t)b^2(t)) \nonumber \\
\geq & \frac{\alpha^2}{6} s^3(t)-\biggl( \frac{\eta^3}{3} s^3(t) + \frac{2}{3\eta^{3/2}} \biggl(\frac{1}{2\gamma^2} (b^*)^2\biggr)^{3/2} \biggr),
\tag{4.30}
\end{align}
where $\eta$ is an arbitrary positive number. Therefore, by (\ref{4-28}) and (\ref{4-29}) we obtain
\begin{align}
\label{4-30}
& \biggl( \frac{\alpha^2}{6}-\frac{\eta^3}{3}\biggr) s^3(t) 
\leq \frac{2}{3\eta^{3/2}} \biggl(\frac{1}{2\gamma^2} (b^*)^2\biggr)^{3/2} \nonumber \\
& + \frac{1}{2}s_0 \biggl |\tilde{u}_0-\frac{b(0)}{\gamma} \biggr|^2_{H} + \frac{\alpha^2}{6} s^3_0 + \frac{1}{2\gamma^2} \int_0^t s^3(\tau) |b_t(\tau)|^2d\tau  \nonumber \\
& +  \frac{b^*}{\gamma^2} \int_0^t s(\tau)|b_t(\tau)|d\tau + \frac{1}{2\beta \gamma^3} \int_0^t |s(\tau)|^2 |b_t(\tau)|^2 d\tau 
\mbox{ for }t\in [0, T]. 
\tag{4.31}
\end{align}
We put $M(\eta)=(s_0|\tilde{u}_0-\frac{b(0)}{\gamma}|^2_H)/2 + (\alpha^2s^3_0)/6+(2(\frac{1}{2\gamma^2} (b^*)^2)^{3/2}))/3\eta^{3/2}$. Then, by taking a suitable $\eta=\eta_0$ and $J_1(t)=\int_0^t s^3(\tau)|b_t(\tau)|^2 d\tau$, $J_2(t)=\int_0^t s(\tau)|b_t(\tau)|d\tau$ and $J_3(t)=\int_0^t |s(\tau)|^2 |b_t(\tau)|^2 d\tau$ for $t\in [0, T]$ we derive that 
\begin{align}
\label{4-31}
\frac{\alpha^2}{12 }s^3(t) \leq M(\eta_0) + \frac{1}{2\gamma^2} J_1(t) + \frac{b^*}{\gamma^2} J_2(t) + \frac{1}{2\beta \gamma^3} J_3(t) 
\mbox{ for }t\in [0, T]. 
\tag{4.32}
\end{align}
Then, we have that 
\begin{align*}
J'_1(t) & \leq \frac{12M(\eta_0)}{\alpha^2}|b_t(t)|^2 + \frac{6}{(\gamma \alpha)^2}|b_t(t)|^2 J_1(t) \nonumber \\
& +\frac{12}{\alpha^2}\biggl (\frac{b^*}{\gamma^2} J_2(t)+ \frac{1}{2\beta \gamma^3} J_3(t) \biggr)|b_t(t)|^2 \mbox{ for }t\in [0, T].
\end{align*}
Hence, by Gronwall's inequality we obtain that 
\begin{align}
\label{4-32}
J_1(t)& \leq \biggl[ \frac{12M(\eta_0)}{\alpha^2}|b_t|^2_{L^2(0, T)} \nonumber \\
& + \frac{12}{\alpha^2}\biggl (\frac{b^*}{\gamma^2} J_2(t) + \frac{1}{2\beta \gamma^3} J_3(t) \biggr)|b_t|^2_{L^2(0, T)}  \biggr] e^{\frac{6}{(\gamma \alpha)^2}|b_t|^2_{L^2(0, T)}} \mbox{ for }t\in [0, T].
\tag{4.33}
\end{align}
Here, we put $N(T)=\frac{12}{\alpha^2}|b_t|^2_{L^2(0, T)}e^{\frac{6}{(\gamma \alpha)^2}|b_t|^2_{L^2(0, T)}}$. Then, by (\ref{4-31}) and (\ref{4-32}) we obtain that 
\begin{align}
\label{4-33}
& \frac{\alpha^2}{12} s^3(t) \leq M(\eta_0) + \frac{1}{2\gamma^2} M(\eta_0)N(T) \nonumber \\
& + \frac{b^*}{\gamma^2} (1+\frac{N(T)}{2\gamma^2})J_2(t) + \frac{1}{2\beta \gamma^3} (1+\frac{N(T)}{2\gamma^2} )J_3(t)
\mbox{ for }t\in [0, T].
\tag{4.34}
\end{align}
Now, we put $l(T)=\max_{0\leq t\leq T}|s(t)|$. Then, we have 
\begin{align*}
& \frac{b^*}{\gamma^2} (1+\frac{N(T)}{2\gamma^2})J_2(t) \leq  \frac{b^*}{\gamma^2} (1+\frac{N(T)}{2\gamma^2})l(T)|b_t|_{L^1(0, T)} \\
\leq & \frac{\eta^3}{3}l^3(T) + \frac{2}{3\eta^{3/2}} \biggl( \frac{b^*}{\gamma^2} (1+\frac{N(T)}{2\gamma^2}) |b_t|_{L^1(0, T)} \biggr)^{3/2},
\end{align*}
and 
\begin{align*}
& \frac{1}{2\beta \gamma^3} (1+\frac{N(T)}{2\gamma^2} )J_3(t) \leq \frac{1}{2\beta \gamma^3} (1+\frac{N(T)}{2\gamma^2}) l^2(T) |b_t|^2_{L^2(0, T)} \\
\leq & \frac{2\eta^{3/2}}{3}l^3(T) + \frac{1}{3\eta^3} \biggl( \frac{1}{2\beta \gamma^3} (1+\frac{N(T)}{2\gamma^2} ) |b_t|^2_{L^2(0, T)} \biggr)^{3}. 
\end{align*}
Hence, by adding these estimates to (\ref{4-33}) and taking a suitable $\eta=\eta_0$ 
%we have that 
%\begin{align*}
%& \frac{1}{2}\biggl( \frac{\alpha^2}{6} - \frac{1}{2\gamma^2} |b_t|^2_{L^2(0, T)}\biggr) L^3(T) \leq \frac{1}{2}s_0 \biggl |\tilde{u}_0-\frac{b(t)}%{\gamma} \biggr|^2_{H} + \frac{\alpha^2}{6} s^3_0 \\
%& + \frac{2}{3\eta^{3/2}_0} \biggl( \frac{1}{2\gamma^2} |b_t|_{L^1(0, T)} \biggr)^{3/2} + \frac{1}{3\eta^3_0} \biggl( \frac{1}{2\beta \gamma^3} |b_t|_{L^1(0, T)} \biggr)^{3} + \frac{2}{3\eta^{3/2}_0} \biggl(\frac{1}{2\gamma^2} (b^*)^2\biggr)^{3/2}. 
%\end{align*}
%Therefore, 
we see that there exists a positive constant $M$ which depends on $\beta$, $\gamma$, $\alpha$, $b^*$, $s_0$, $|b_t|_{L^2(0, T)}$ and $|b_t|_{L^1(0, T)}$ such that $s(t)\leq M$ for $t \in [0, T]$. 
%Next, we show (ii). Finally, This result is proved by repeating the argument of the proof of Lemma \ref{lem9} replaced $L(T)$ and $u^*(T)$ by $M$ and $u^*$, respectively, we see that (ii) holds. Thus, Lemma \ref{lem7} is proven.
%\end{proof}
%Finally, from this result of the maximal length of the free boundary $s$ and (\ref{4-28}) we can see that (i) holds. 
%From this result and the integration (\ref{4-19-5}) over $[0, t]$, we see that there exists a positive constant $M$ such that (i) holds. 

%\begin{comment}
Next, we show (ii). This result is proved by the argument of the proof of Lemma \ref{lem9} replaced $l(T)$ and $u^*(T)$ by $M$ and $u^*$, respectively. We put $U(t, y)=[\tilde{u}(t, y)-u^*]^+$ for $y \in [0, 1]$ and $t\in [0, T]$. Then, by (\ref{4-15-0}) it holds that 
\begin{align}
\label{4-34}
& \frac{1}{2} \frac{d}{dt}|U(t)|^2_{H}+ \frac{1}{s^2(t)} \int_0^1 |U_y(t)|^2 dy + \frac{1}{s(t)} \tilde{u}(t, 1)s_t(t)U(t, 1) \nonumber \\
& -\frac{1}{s(t)}\beta (b(t)-\gamma \tilde{u}(\cdot, 0))U(t, 0)= \int_0^1\frac{y s_t(t)}{s(t)} \tilde{u}_y(t)U(t) dy
\mbox{ for a.e. }t\in [0, T].
\tag{4.35}
\end{align}
Here, by $\tilde{u}(t)\geq 0$ on $[0, 1]$ for $t\in [0,T]$ we note that  $s_t(t)=a_0(\sigma(\tilde{u}(t, 1))-\alpha s(t))=a_0(\tilde{u}(t, 1)-\alpha s(t))$. Then, we observe that 
\begin{align*}
& \int_0^1\frac{y s_t(t)}{s(t)} \tilde{u}_y(t)U(t) dy = \frac{s_t(t)}{2s(t)} \int_0^1 \biggl (\frac{d}{dy}(yU^2(t))-U^2(t) \biggr) dy \\
& =\frac{s_t(t)}{2s(t)}|U(t, 1)|^2 -\frac{s_t(t)}{2s(t)}|U(t)|^2_H \leq \frac{s_t(t)}{2s(t)}|U(t, 1)|^2 + \frac{a_0\alpha }{2}|U(t)|^2_H.
\end{align*}
Similarly to (\ref{3-45}), the forth term in the left-hand side is non-negative so that 
\begin{align}
\label{4-35}
& \frac{1}{2} \frac{d}{dt}|U(t)|^2_{H}+ \frac{1}{s^2(t)} \int_0^1 |U_y(t)|^2 dy + \frac{1}{s(t)} \tilde{u}(t, 1)s_t(t)U(t, 1)-\frac{s_t(t)}{2s(t)}|U(t, 1)|^2  \nonumber \\
& \leq \frac{a_0\alpha}{2}|U(t)|^2_H
\mbox{ for a.e. }t\in [0, T].
\tag{4.36}
\end{align}
Here, by $u^*\geq \alpha M \geq \alpha s(t)$ for $t\in [0, T]$, it holds that 
\begin{align*}
& \frac{s_t(t)}{s(t)}\tilde{u}(t, 1)U(t, 1)-\frac{s_t(t)}{2s(t)}|U(t, 1)|^2 = \frac{s_t(t)}{2s(t)}|U(t, 1)|^2 + \frac{s_t(t)}{s(t)}u^*U(t, 1)\\
\geq & \frac{a_0(u^*-\alpha s(t))}{s(t)}(\frac{|U(t, 1)|^2}{2}+ u^*U(t, 1)) \geq 0.
\end{align*}
%For the third term in the left-hand side of (\ref{4-20}) we can write 
%By noting from $\sup_{r\in \mathbb{R}}\varphi(r)\leq |h|_{L^{\infty}(0, T)}H^{-1}$ in (A4) that
%\begin{align*}
%& -u_z(t, s(t))U(t, s(t)) = 
%\frac{1}{s(t)}\tilde{u}(t, 1)s_t(t) U(t, 1) & = \frac{1}{s(t)}(s_t(t) \left(\tilde{u}(t, 1)-u^* \right)U(t, 1) + s_t(t) u^* U(t, 1)) \\  
%& =\frac{s_t(t)}{s(t)}|U(t, s(t))|^2 + \frac{s_t(t)}{s(t)}u^* U(t, s(t)).
%\end{align*}
%\begin{align*}
%& \int_0^1\frac{y s_t(t)}{s(t)} \tilde{u}_y(t)U(t) dy = \frac{s_t(t)}{2} \int_0^1 \biggl (\frac{d}{dy}(yU^2(t))-U^2(t) \biggr) dy \\
%& =\frac{s_t(t)}{2}|U(t, 1)|^2 -\frac{s_t(t)}{2}|U(t)|^2_H 
%\end{align*}
%by (\ref{1-3}) and $b\leq b^*$, we observe that 
%\begin{align*}
%& -\frac{1}{s(t)}\beta(b(t)-\gamma \tilde{u}(t, 0))U(t, 0) = \frac{1}{s(t)}\beta (\gamma \tilde{u}(t, 0)-b^* + b^*-b(t))U(t, 0) \\
%\geq  & \frac{\beta \gamma}{s(t)} |U(t, 0)|^2 + \frac{\beta}{s(t)}(b^*-b(t))U(t, 0) \geq 0.
%\end{align*}
By applying the result to (\ref{4-35}) we obtain that 
\begin{align*}
& \frac{1}{2}\frac{d}{dt}\int_0^1|U(t)|^2 dy + \frac{1}{2s^2(t)}\int_0^1 |U_y(t)|^2 dy \leq \frac{a_0\alpha}{2}|U(t)|^2_H \mbox{ for a.e. }t\in[0, T].
%\label{4-36}
\end{align*}
%Here, by $s_t(t)=a_0(\sigma(u(t, s(t)))-\alpha s(t))$ we notice that 
%\begin{align*}
%\frac{s_t(t)}{2}|U(t, s(t))|^2 \geq \frac{a_0(u^*-\alpha s(t))}{2}|U(t, s(t))|^2 \geq \frac{a_0(u^*-\alpha M)}{2}|U(t, s(t))|^2\geq 0.
%\end{align*}
%This implies that 
This implies that $\tilde{u}(t) \leq u^*$ on $[0, 1]$ for $t\in [0, T]$. 
%\end{comment}
Thus, Lemma \ref{lem12} is now proven.
\end{proof}
%\end{comment}

Combining the statements of Lemma \ref{lem11} and Lemma \ref{lem12}, we can conclude that Theorem \ref{t2} holds.

\section*{Acknowledgements} The first author is supported by Grant-in-Aid No. 20K03704, JSPS. 
The second author is supported by Grant-in-Aid No. 19K03572, JSPS.

\medskip
\end{document}